\documentclass[12pt]{amsart}
 \usepackage{amsfonts,amssymb,enumerate,color,bm,hhline,makecell,array}
\usepackage{array}
\usepackage{mathrsfs}
\usepackage[all,ps,cmtip]{xy} 
\usepackage[normalem]{ulem}

\usepackage{enumitem}
\setlist[itemize]{noitemsep,nolistsep}
\setlist[enumerate]{noitemsep,nolistsep}
\usepackage{colortbl} 

   \allowdisplaybreaks

\let\mathcal\mathscr

\def\C{{\bf C}}

\def\N{{\bf N}}
\def\R{{\bf R}}
\def\Q{{\bf Q}}
\def\P{{\bf P}}

\def\Z{{\bf Z}}

\def\L{{\Lambda}}

\def\hKm{hyperk\"ahler manifold}

\def\bbw#1#2{\textstyle{\bigwedge\hskip-0.9mm^{#1}}\hskip0.2mm{#2}}

\def\phi{{\varphi}}

\def\cC{\mathcal{C}}
\def\cD{\mathcal{D}}

\def\cF{\mathcal{F}}
\def\cFlop{{\cF\hskip-.2mm  lop}}

\def\cH{\mathcal{H}}

\def\cK{\mathcal{K}}

\def\cM{\mathcal{M}}

\def\cO{\mathcal{O}}
\def\cP{\mathcal{P}}

\def\cU{\mathcal{U}}

\def\lra{\longrightarrow}
\def\llra{\hbox to 10mm{\rightarrowfill}}
\def\lllra{\hbox to 15mm{\rightarrowfill}}

\def\llla{\hbox to 10mm{\leftarrowfill}}
\def\lllla{\hbox to 15mm{\leftarrowfill}}

\def\dra{\dashrightarrow}
\def\thra{\twoheadrightarrow}
\def\hra{\hookrightarrow}
\def\isom{\simeq}
\def\eps{\varepsilon}
\def\ie{\hbox{i.e.}}

\def\tO{\widetilde{O}}

 \def\vide{\varnothing}

\DeclareMathOperator{\isomto}{\stackrel{{}_{\scriptstyle\sim}}{\to}}
\DeclareMathOperator{\isomdra}{\stackrel{{}_{\scriptstyle\sim}}{\dra}}
\DeclareMathOperator{\isomlra}{\stackrel{{}_{\scriptstyle\sim}}{\lra}}

\DeclareMathOperator{\Amp}{Amp}
\DeclareMathOperator{\Aut}{Aut}
\DeclareMathOperator{\Bir}{Bir}

\DeclareMathOperator{\disc}{disc}
\def\div{\mathop{\rm div}\nolimits}

 \DeclareMathOperator{\Exc}{Exc}

\DeclareMathOperator{\Gr}{Gr}

\DeclareMathOperator{\Hom}{Hom}
\DeclareMathOperator{\Id}{Id}
\DeclareMathOperator{\Int}{Int}
\def\Im{\mathop{\rm Im}\nolimits}

\DeclareMathOperator{\Ker}{Ker}
\DeclareMathOperator{\KKK}{K3}

\DeclareMathOperator{\Mov}{Mov}
\DeclareMathOperator{\Nef}{Nef}

\DeclareMathOperator{\Pic}{Pic}
\DeclareMathOperator{\Pos}{Pos}

\def\llra{\hbox to 10mm{\rightarrowfill}}
\def\lllra{\hbox to 15mm{\rightarrowfill}}

\newtheorem{lemm}{Lemma}[section]
\newtheorem{theo}[lemm]{Theorem}
\newtheorem{coro}[lemm]{Corollary}
\newtheorem{prop}[lemm]{Proposition}

\theoremstyle{definition}

\newtheorem{rema}[lemm]{Remark}

\newtheorem{exam}[lemm]{Example}

\theoremstyle{remark}
\newtheorem*{remark*}{Remark}
\newtheorem*{note*}{Note}

\begin{document}
\title[Period map for hyperk\"ahler fourfolds]{On the period map for polarized hyperk\"ahler fourfolds}

\author[O.~Debarre]{Olivier Debarre}
\address{\parbox{0.9\textwidth}{Univ  Paris Diderot, \'Ecole Normale Su\-p\'e\-rieu\-re, PSL Research University,
\\[1pt] 
CNRS, D\'epar\-te\-ment Math\'ematiques et Applications
\\[1pt]
45 rue d'Ulm, 75230 Paris cedex 05, France}}
\email{{olivier.debarre@ens.fr}}

\author[E.~Macr\`i]{Emanuele Macr\`i}
\address{\parbox{0.9\textwidth}{Northeastern University,
\\[1pt] 
Department of Mathematics
\\[1pt]
360 Huntington Avenue, Boston, MA 02115, USA}}
\email{{e.macri@northeastern.edu}}


\thanks{E.~M.~ was partially supported by the NSF grants DMS-1523496, DMS-1700751, and by a Poincar\'e Chair from the Institut Henri Poincar\'e and the Clay Mathematics Institute.}

\begin{abstract}
We study smooth projective hyperk\"ahler fourfolds that are deformations of Hilbert squares of K3 surfaces and are equipped with a polarization of fixed degree and divisibility.\ They are parametrized by a quasi-projective irreducible 20-dimensional moduli space and  Verbitksy's Torelli theorem implies that their period map is an open embedding.\ 

Our main result is that  the complement of the image of the period map is a finite union of explicit Heegner divisors that we describe.\ We also  prove that   infinitely many Heegner divisors in a given period space have the property that  their general points correspond to fourfolds which are isomorphic to   Hilbert squares of a K3 surfaces, or to double EPW sextics. 

In two appendices, we determine the groups of biregular or birational automorphisms of various projective hyperk\"ahler fourfolds with Picard number 1 or 2.
\end{abstract}

\maketitle

\setcounter{tocdepth}{1}

\section{Introduction}
We consider smooth projective hyperk\"ahler fourfolds $X$  which are deformations of Hilbert squares of K3 surfaces (one says that  $X$ is of $\KKK^{[2]}$-type).\  The abelian group $H^2(X,\Z)$ is free of rank 23 and it is equipped the Beauville--Bogomolov--Fujiki form $q_X$, a non-degenerate $\Z$-valued quadratic form of signature $(3,20)$ (\cite[Th\'eor\`eme~5]{beaa}).\  A polarization $H$ on $X$ is the class of an ample line bundle on $X$ that is primitive (\ie, non-divisible) in the group $H^2(X,\Z)$.\  The {\em square} of $H$ is the  positive even integer $2n:=q_{X}(H)$ and its {\em divisibility} is the integer $\gamma\in\{1,2\}$ such that $H\cdot H^2(X,\Z)=\gamma\Z$ (the case $\gamma=2$  only occurs when $n\equiv -1\pmod4$).

Smooth polarized hyperk\"ahler fourfolds $(X,H)$  of $\KKK^{[2]}$-type of degree $ 2n$ and divisibility $\gamma$ admit an irreducible quasi-projective coarse moduli space $\cM^{(\gamma)}_{2n}$ of dimension $20$.\  The period map (see Section~\ref{sec21}) 
\begin{equation*}
\wp^{(\gamma)}_{2n}\colon\cM^{(\gamma)}_{2n}\lra  \cP^{(\gamma)}_{2n}
\end{equation*}
is algebraic  and it is an open embedding  by Verbitsky's Torelli  Theorem~\ref{tor1}.\  Our main result is that the image of $ \wp^{(\gamma)}_{2n}$ is the complement of a finite union of   {\em Heegner divisors} (this can also be deduced  from the general results in \cite{amve}) which can be explicitly listed (Theorem~\ref{imper}).
 
The main ingredient in the proof is the explicit determination of the nef   and movable cones of  smooth projective hyperk\"ahler fourfolds of $\KKK^{[2]}$-type (see Theorem~\ref{thm:NefConeHK4}).\ This is a simple consequence of previous results by Markman~(\cite{marsur}), Bayer--Macr\`i~(\cite{bama}), Bayer--Hassett--Tschinkel~(\cite{bht}), and Mongardi~(\cite{mon2}).

The \emph{Noether--Lefschetz locus} is the inverse image  by the period map in $\cM^{(\gamma)}_{2n}$ of the union of all Heegner divisors.\ Its irreducible components  were shown in~\cite[Theorem~1.5]{ber} to generate (over $\Q$) the Picard group of $\cM^{(\gamma)}_{2n}$.\ As an application of our Theorem~\ref{thm:NefConeHK4}, we study in Section~\ref{sec9}  birational isomorphisms between some of these components.\ In particular, we show that points corresponding to Hilbert squares of K3 surfaces are dense in the moduli spaces $\cM^{(\gamma)}_{2n}$.
 
In the two appendices, we collect results on  biregular and birational automorphisms of certain projective hyperk\"ahler fourfolds with Picard number 1 or 2.\
These results are needed in some of the arguments in Section~\ref{sec9}.

Since   the nef and movable cones can be described in all dimensions, many of our results extend with some modifications to smooth projective hyperk\"ahler manifolds of $\KKK^{[n]}$-type.\
More details in the higher dimensional case will appear in~\cite{debsurvey}.

\subsection*{Acknowledgements.}
We would like to thank Ekaterina Amerik, Arend Bayer, Samuel Bois\-si\`ere, Christian Lehn, Eyal Markman, Kieran O'Grady, and Emmanuel Ullmo for useful discussions and suggestions.

\section{Lattices}\label{sec1}

A lattice is a free abelian group $\L$ of finite rank endowed with a $\Z$-valued non-degenerate quadratic form~$q$.\ It is {\em even} if  $q$ only takes even values.\ We extend $q$ to a $\Q$-valued quadratic form on $\L\otimes\Q$, hence also on the dual
$$\L^\vee:=\Hom_\Z(\L,\Z)=\{x\in \L\otimes\Q\mid \forall y\in\L\quad x\cdot y\in\Z\}.$$
The {\em discriminant group} of $\L$ is the finite abelian group 
$$D(\L):=\L^\vee/\L.$$
 The lattice $\L$ is {\em unimodular} if the group $D(\L)$ is trivial.\ If $x$ is a non-zero element of $\L$, we define the integer $\div_\L (x)$ (the    {\em divisibility} of $x$) as the positive generator of the subgroup $x\cdot\L$ of $\Z$.\ We also consider $ x/\!\div_\L(x)$, a primitive (\ie, non-zero and non-divisible) element of $\L^\vee$, and its class $x_*=[x/\!\div_\L(x)]$ in the group $ D(\L)$, an element of order $\div_\L(x)$.
 
If $t$ is a non-zero integer, we let $\L(t)$ be the lattice $(\L,tq)$.\ We let $I_1$ be the lattice $\Z$ with the quadratic form
$q(x)=x^2$ and we let $U$ (the hyperbolic plane) be the even unimodular lattice $\Z^2$ with the quadratic form
$q(x_1,x_2)=2x_1x_2$.\ There is a unique positive definite even unimodular lattice of rank~8, which we denote by $E_8$.\ 

{\em Assume now that the lattice $\L$ is even.}\  Following \cite{nik}, we
define a quadratic form $\bar q\colon D(\L)\to \Q/2\Z$ 
by setting $\bar q([x]):=q(x) \in \Q/2\Z$.\ The {\em stable orthogonal group} $\tO(\L,q)$ is the kernel of  the canonical map
$$O(\L,q)\lra O(D(\L),\bar q).$$
This map is surjective when $\L$ is indefinite and its rank is at least the minimal number of generators of the finite abelian group $D(\L)$ plus 2 (\cite[Theorem~1.14.2]{nik}).

We will  use  the following classical result  (see~\cite[Lemma~3.5]{ghs}).\ 

\begin{theo}[Eichler's criterion]\label{eic}
Let $\Lambda$ be an even lattice that contains at least two orthogonal copies of $U$.\ The $\tO(\L,q)$-orbit of a  primitive vector  $x\in\L$ is determined by the integer $q(x)$ and the element $x_*$ of $D(\L)$.
\end{theo}

\section{Moduli spaces, period spaces, and period maps}\label{sec2}

\subsection{Moduli spaces} Let $X$ be a (smooth) hyperk\"ahler (also called irreducible symplectic) fourfold    of $\KKK^{[2]}$-type (see \cite[Section~3]{ghssur} for the main definitions).\ The lattice $(H^2(X,\Z),q_{X})$ defined in the introduction is isomorphic to
the even lattice 
$$\L_{\KKK^{[2]}}:=U^{\oplus 3}\oplus E_8(-1)^{\oplus 2}\oplus I_1(-2)$$
with signature $(3,20)$ and discriminant group  $\Z/2\Z$.\ The divisibility of a primitive element is therefore 1 or 2, and $\tO(\L_{\KKK^{[2]}})=O(\L_{\KKK^{[2]}})$.

As recalled in the introduction, hyperk\"ahler fourfolds $X$ of $\KKK^{[2]}$-type with a polarization of  fixed (positive) square $2n$ and divisibility $\gamma $  in the lattice $(H^2(X,\Z),q_X)$ have a quasi-projective coarse moduli space  $\cM^{(\gamma)}_{2n}$  which is  irreducible  and 20-dimensional  when $\gamma=1$, or when $\gamma=2$ and $n \equiv -1 \pmod4$ (\cite[Remark~3.17]{ghssur}).

\subsection{Period spaces and period maps}\label{sec21}

By Eichler's criterion (Theorem~\ref{eic}), primitive elements   of the lattice  $\L_{\KKK^{[2]}}$  with fixed positive square $2n$ and fixed divisibility $\gamma\in\{1,2\}$ form a single $O(\L_{\KKK^{[2]}})$-orbit.\ We fix one such element~$h_0$.\ 
  
 If $\gamma=1$, we have
\begin{equation}\label{lats}
h_0^\bot \isom    U^{\oplus 2}\oplus E_8(-1)^{\oplus 2}\oplus I_1(-2)\oplus I_1(-2n)=:\L^{(1)}_{\KKK^{[2]},2n},
\end{equation}  
a lattice with discriminant 
group $ \Z/2n\Z \times\Z/2 \Z$, with   $\bar q(1,0)=-\frac{1}{2n}$ and $\bar q(0,1)=-\frac12$  (see the proof of Proposition \ref{propirr}). 

If $\gamma=2$, we have $n \equiv -1\pmod4$ and
\begin{equation}\label{latns}
h_0^\bot \isom  U^{\oplus 2}\oplus E_8(-1)^{\oplus 2}\oplus \begin{pmatrix}-2&-1\\-1&-\tfrac{n+1}{2}\end{pmatrix}=:\L^{(2)}_{\KKK^{[2]},2n},
\end{equation}
a lattice  with discriminant 
group $\Z/n\Z$, with  $\bar q(1)=-\frac{2}{n}$ (see the proof of Proposition \ref{propirr}).

We now describe the period map for polarized hyperk\"ahler fourfolds   of $\KKK^{[2]}$-type.\ The complex variety 
$$\Omega_{h_0} := \{ x\in \P(\L_{\KKK^{[2]}} \otimes \C)\mid  x\cdot h_0=0,\ x\cdot x=0,\ x\cdot \bar x>0\}$$
 has two connected components, interchanged by complex conjugation, which are Hermitian symmetric domains of
type IV.\ It is acted on by the   group 
 $$ O(\L_{\KKK^{[2]}},h_0):=\{\Phi\in  O(\L_{\KKK^{[2]}})\mid \Phi(h_0)=h_0\}.$$
By results of Baily--Borel and Griffiths, the quotient
$\cP^{(\gamma)}_{2n} :=O(\L_{\KKK^{[2]}},h_0)\backslash \Omega_{h_0}$
 is an irreducible quasi-projective variety and the {\em period map}
 \begin{equation}\label{defp}
 \wp^{(\gamma)}_{2n}\colon\cM^{(\gamma)}_{2n}\lra  \cP^{(\gamma)}_{2n}
 \end{equation}
is algebraic.\ Alternatively,  
one has
$$\Omega_{h_0} \isom \{ x\in \P(\L^{(\gamma)}_{\KKK^{[2]},2n} \otimes \C)\mid    x\cdot x=0,\ x\cdot \bar x>0\}$$
and   the group $O(\L_{\KKK^{[2]}},h_0)$ can be identified with the stable orthogonal group 
  $\widetilde O(\L^{(\gamma)}_{\KKK^{[2]},2n})$  (\cite[Proposition 3.12 and Corollary 3.13]{ghs}).
  
The full orthogonal group  $O(\L^{(\gamma)}_{\KKK^{[2]},2n})$ also acts on $\Omega_{h_0}$, hence the quotient group
$$O(\L^{(\gamma)}_{\KKK^{[2]},2n})/\widetilde O(\L^{(\gamma)}_{\KKK^{[2]},2n})\isom O(D(\L^{(\gamma)}_{\KKK^{[2]},2n}))$$
 acts on the period space $ \cP^{(\gamma)}_{2n}$ (where  $-\Id$   acts trivially).\ We determine this group and describe this action in the next proposition, assuming for simplicity that $n$ is odd.\ For any non-zero integer $r$, we denote by $\rho(r)$ the number of prime factors of~$r$.

\begin{prop}\label{invol}
Assume that $n$ is odd.\ The period space $ \cP^{(\gamma)}_{2n}$ is acted on generically freely by the following groups:
\begin{itemize}
\item if   $n\equiv 1\pmod4$ (so that $\gamma=1$), by  the group $(\Z/2\Z)^{\max\{\rho(n),1\}}$;
\item if $n\equiv -1\pmod4$, by the group $(\Z/2\Z)^{\rho(n)-1}$.
\end{itemize}
\end{prop}

\begin{proof}
{\bf Case $\gamma=1$}.\ Since $n$ is odd, \eqref{lats} implies $D(\L^{(\gamma)}_{\KKK^{[2]},2n}) \isom \Z/n\Z\times  \Z/2\Z\times \Z/2\Z$.\ This decomposition is still orthogonal for $\bar q$ and the
 values of $\bar q$ at the points of order 2 are $\bar q(0,0,1)=-\frac{1}{2}$, $\bar q(0,1,0)=-\frac{n}{2}$, and $\bar q(0,1,1)=-\frac{n+1}{2}$ in $\Q/2\Z$.
        
When $n\equiv -1\pmod4$, these three values are all distinct and any isometry $\Phi$ of $ ( D(\L^{(1)}_{\KKK^{[2]},2n}),\bar q)$ must therefore be the identity on both $\Z/2\Z$ factors, hence must preserve their orthogonal $\Z/n\Z $.\  Write $\Phi(1,0,0 )=(a,0,0 )$; since $\bar q(1,0,0 )=-\frac{2}{n}$, we have  $\frac{2}{n}=\frac{2a^2}{n}\pmod{2\Z}$, hence $a^2=1\pmod{n}$.\ Since $n$ is odd, the group $O(D(\L^{(1)}_{\KKK^{[2]},2n}))$ is therefore isomorphic to $(\Z/2\Z)^{\rho(n)}$.\ The proposition follows since only $-\Id$ acts trivially.

When $n\equiv 1\pmod4$, there are extra isometries given by $(0,1,0)\leftrightarrow (0,0,1 )$, $(1,0,0 )\mapsto (a,0,0 )$, where $a^2\equiv 1\pmod{n}$ (and when $n=1$, we have $-\Id=\Id$).\ 

\medskip
{\bf Case $\gamma=2$} (hence $n\equiv -1\pmod4$).\ We have $D(\L^{(\gamma)}_{\KKK^{[2]},2n}) \isom \Z/n\Z$, with $\bar q(1)= -\frac{2}{n}$, and we proceed as in   the first case.\ This proves the proposition.
\end{proof}

\subsection{The Torelli theorem}\label{sec21a}

Verbitsky's Torelli theorem   takes  the following form (\cite{ver},~\cite[Theorem~3.14]{ghssur}).

\begin{theo}[Verbitsky]\label{tor1}
For each positive integer $ n$ and each divisibility $\gamma\in\{1,2\}$, the period map 
 \begin{equation*} 
 \wp^{(\gamma)}_{2n}\colon\cM_{2n}^{(\gamma)}\lra  \cP^{(\gamma)}_{2n}
 \end{equation*}
 is an open embedding. 
\end{theo}

In particular, the commuting involutions of $\cP^{(\gamma)}_{2n}$ described in Proposition~\ref{invol} induce rational involutions on the moduli space $\cM_{2n}^{(\gamma)}$.\ 
In the case $n=\gamma=1$ (double EPW sextics; see Example~\ref{qqq1}), the unique non-trivial involution was described geometrically in~\cite{og2} in terms of projective duality.

\section{Special polarized hyperk\"ahler fourfolds}\label{sec22}

A hyperk\"ahler fourfold corresponding to a very general point of $\cM^{(\gamma)}_{2n}$ has Picard number~1.\ The {\em Noether--Lefschetz locus} (or special locus) is the subset of $\cM^{(\gamma)}_{2n}$ corresponding to hyperk\"ahler fourfolds with Picard number at least 2.\ It can be described as follows.

Let $K$ be a primitive, rank-2, signature-$(1,1)$ sublattice of $\L_{\KKK^{[2]}}$ containing the class $h_0$ chosen in Section~\ref{sec21}.\ The codimension-2 subspace $\P(K^\bot\otimes \C)$ in $\P(\L_{\KKK^{[2]}} \otimes \C)$ cuts out an irreducible hypersurface in $\Omega_{h_0}$ whose image in $\cP^{(\gamma)}_{2n}$ will be denoted by $\cD_{2n,K}^{(\gamma)}$ and called a {\em Heegner divisor}.\  
The Noether--Lefschetz locus  is then the inverse image in $\cM^{(\gamma)}_{2n}$ by the period map   $\wp^{(\gamma)}_{2n}$ of the countable union $\bigcup_K \cD_{2n,K}^{(\gamma)}$ of irreducible hypersurfaces.

For each   integer $d$, the union  
$$\cD_{2n,d}^{(\gamma)}:=\hskip-5mm\bigcup_{\disc(K^\bot)=-d}\hskip-5mm \cD_{2n,K}^{(\gamma)} \subset \cP_{2n}^{(\gamma)}
$$ 
of Heegner divisors is 
finite, hence it is either empty or of pure codimension 1.\ Following Hassett, we say that the polarized hyperk\"ahler fourfolds whose period point is in $\cD_{2n,d}^{(\gamma)} $ are  {\em special of discriminant $d$} (the lattice $K^\bot$ has signature $(2,19)$, hence  $d  $ is positive).\ 
We use the notation $\cC_{2n,d}^{(\gamma)}:=(\wp^{(\gamma)}_{2n})^{-1}(\cD_{2n,d}^{(\gamma)}) \subset \cM^{(\gamma)}_{2n}$.

We now describe the irreducible components of  the loci $\cD_{2n,d}^{(\gamma)}$ (the case $n=3$ and $\gamma=2$ was originally studied by Hassett in~\cite{has} and the case $n=\gamma=1$ in~\cite{dims}). 

\begin{prop}\label{propirr}
Let $n$ and $d$ be positive integers and let $\gamma\in\{1,2\}$.\ If the locus $\cD_{2n,d}^{(\gamma)}$ is non-empty, the integer $d$ is even; we set $e:=d/2$.

\noindent {\rm (1)\hskip 4.8mm  (a)} The locus $\cD_{2n,2e}^{(1)}$ is non-empty if and only if   either $e$ or $e-n$ is a  square modulo~$4n$.
 
 {\rm (b)}  If $n$ is {\em square-free} and $e$ is divisible by $n$ and   satisfies the conditions in {\rm (a)}, 
  the locus $\cD_{2n,2e}^{(1)}$ is irreducible, except when 
 \begin{itemize}
\item either $n\equiv 1\pmod4$ and $e\equiv  n\pmod{4n}$,
\item or   $n\equiv -1\pmod4$ and $e\equiv 0\pmod{4n}$,
\end{itemize}
  in which cases $\cD_{2n,2e}^{(1)}$ has two irreducible components. 

{\rm (c)} If  $n$ is {\em prime} and $e$ satisfies the conditions in {\rm (a)}, $\cD_{2n,2e}^{(1)}$ is  irreducible, except when $n\equiv 1\pmod4$ and $e\equiv  1  \pmod{4}$,  or when $n\equiv -1\pmod4$ and  $e\equiv  0  \pmod{4}$,  in which cases  $\cD_{2n,2e}^{(1)}$ has two irreducible components.

\noindent {\rm (2)} Assume moreover $ n\equiv -1\pmod4$.

{\rm (a)}   The locus $\cD_{2n,2e}^{(2)}$ is non-empty if and only if    $e$ is a  square modulo~$n$.

{\rm (b)} If $n$ is {\em square-free and $ n\mid e$,} the locus $\cD_{2n,2e}^{(2)}$ is  irreducible.

{\rm (c)} If $n$ is {\em prime} and $e$ satisfies the conditions in {\rm (a)}, $\cD_{2n,2e}^{(2)}$ is irreducible.
\end{prop}

\begin{rema}\label{exc}
In cases (1)(b) and (1)(c), when the hypersurface $\cD_{2n,2e}^{(1)}$ is reducible, its two components are exchanged by one of the involutions of the period space  described in Proposition~\ref{invol}  when $n\equiv 1\pmod4$, but not when $n\equiv -1\pmod4$ (in that case, these involutions are in fact trivial when $n$ is prime).
\end{rema}

\begin{proof}[Proof of Proposition \ref{propirr}]
\noindent{\bf Case $\gamma=1$.}\ 
 Let  $(u,v)$ be a standard basis for a hyperbolic plane $U$ contained in $\L_{\KKK^{[2]}}$ and let~$\ell$ be a basis for the $I_1(-2)$ factor.\ We may take  
$h_0:=u+nv
$ (it has the correct square and divisibility), in which case  $h_0^\bot=\Z( u-nv)\oplus\Z\ell\oplus M$, where  $M:=\{u,v,\ell\}^\bot=U^{\oplus 2}\oplus E_8(-1)^{\oplus 2}$ is unimodular.\ The   discriminant group  $D(h_0^\bot)\isom \Z/2\Z\times \Z/2n\Z$ is generated by $\ell_*=\ell/2$ and $(u-nv)_*=(u-nv)/2n$, with $\bar q(\ell_*)=-1/2$ and  $\bar q((u-nv)_*)=-1/2n$.

Let $\kappa$ be a generator of $K\cap h_0^\perp$.\ We write
$$\kappa=a(u-nv)+b\ell +cw,
$$
where $w\in M$ is primitive.\ Since $K$ has signature $(1,1)$, we have $\kappa^2<0$ and  the formula   from~\cite[Lemma~7.5]{ghssur}    reads
\begin{equation}\label{eqd}
d=\left|\disc(K^\bot)\right|=\left|\frac{\kappa^2\disc(h_0^\bot)}{s^2}\right|  = \frac{8n(na^2+b^2+mc^2)}{s^2}\equiv \frac{8n(na^2+b^2)}{s^2}\pmod{8n},
\end{equation}
where $m:=-\frac12 w^2$ and $s: =\gcd(2na,2b,c) $ is the divisibility of $\kappa$ in $ h_0^\bot$.\ If $s\mid b$, we obtain $d\equiv 2 \bigl( \frac{2na}{s}\bigr)^2\pmod{8n}$, which is the first case of (1)(a): $d$ is even and $e:=d/2$ is a square modulo $4n$.\ Assume  $s\nmid b$ and, for any non-zero integer $x$, write $x=2^{v_2(x)}x_{\textnormal{odd}}$, where $x_{\textnormal{odd}}$ is odd.\ One has then
$\nu_2(s)=\nu_2(b)+1
$ and
$$
d\equiv 2 \Bigl( \frac{2na}{s}\Bigr)^2+2n\Bigl( \frac{b_{\textnormal{odd}}}{s_{\textnormal{odd}}}\Bigr)^2 \equiv 2 \Bigl( \frac{2na}{s}\Bigr)^2+2n   \pmod{8n},
$$
which is the second case of (1)(a): $d$ is even and $d/2 -n$ is a square modulo $4n$.\ It is then easy, taking suitable integers $a$, $b$,~$c$, and vector $w$, to construct examples that show that these necessary conditions on $d$ are also sufficient, thereby proving (1)(a).

We now prove (1)(b) and (1)(c).\ 
 
Given a lattice $K$ containing $h_0$ with $\disc(K^\perp)=-2e$, we let as above $\kappa$ be a generator of $K\cap h_0^\perp$.\  By Eichler's criterion (Theorem~\ref{eic}),   the group $\widetilde O(h_0^\bot)$ acts transitively on the set of primitive vectors $\kappa\in h_0^\bot$ of given square and fixed $\kappa_*\in D(h_0^\bot )$.\ Since $\kappa$ and $-\kappa$ give rise to the same lattice $K$ (obtained as the saturation of $\Z h_0\oplus \Z \kappa$), the locus $ \cD^{(1)}_{2n,2e} $ will be irreducible (when non-empty) if we show that
  the integer $e$ determines $\kappa^2$, and $\kappa_*$ up to sign.
 
We write as above $\kappa=a(u-nv)+b\ell +cw\in h_0^\bot$, with $\gcd(a,b,c)=1$ and    $s =\div_{h_0^\bot}(\kappa)=\gcd(2na,2b,c)$.\ From \eqref{eqd}, we get
\begin{equation}\label{kap}
\kappa^2=-es^2/2n=-2(na^2+b^2+mc^2)\quad\textnormal{and}\quad
\kappa_*= (2na/s,2b/s)\in \Z/2n\Z\times \Z/2\Z.
\end{equation}

If  $s=1$, we have $e \equiv 0\pmod{4n}$ and $\kappa_*=0$.
 
 If $s=2$, the integer $c$ is even and $a$ and $b$ cannot be both even (because $\kappa$ is primitive).\ We have $e=n(na^2+b^2+mc^2)$ and
 $$ \begin{cases} 
 e \equiv  n^2 \pmod{4n}\textnormal{\quad and\ }\kappa_*=(n,0) &\textnormal{if $b$ is even (and $a$ is odd);}\\
  e \equiv  n \pmod{4n}\textnormal{\quad and\ }\kappa_*=(0,1) &\textnormal{if $b$ is odd and $a$ is even;}\\
   e \equiv  n(n+1) \pmod{4n}\textnormal{\quad and\ }\kappa_*=(n,1) &\textnormal{if $b$ and $a$ are odd.}
 \end{cases}
 $$ 
 
 Assume now that $n$ is square-free and $ n\mid e$.\  From \eqref{eqd}, we get $n\mid \bigl( \frac{2na}{s}\bigr)^2$, hence $s^2\mid 4na^2$, and 
 $s\mid 2a$ because $n$ is square-free.\ This implies $s=\gcd(2a,2b,c)\in\{1,2\}$.\ 
 
 When $n$ is even (\ie, $n\equiv 2\pmod4$),  we see from the discussion above that  both $s$ (hence also $\kappa^2$) and $\kappa_*$ are determined by $e$, so the corresponding hypersurfaces $ \cD^{(1)}_{2n,2e} $ are  irreducible.\ 
 
 If $n$ is odd, there are coincidences: 
 \begin{itemize}
 \item when $n\equiv 1\pmod4$, we have $ n\equiv  n^2 \pmod{4n}$, hence $ \cD^{(1)}_{2n,2e} $ is  irreducible when $e\equiv 0$ or $2n\pmod{4n}$,   has two irreducible components (corresponding to $\kappa_*=(n,0)$ and $\kappa_*=(0,1)$) when $e\equiv  n\pmod{4n}$, and is empty otherwise; 
 \item when $n\equiv -1\pmod4$, we have $ n(n+1)\equiv 0 \pmod{4n}$, hence $ \cD^{(1)}_{2n,2e} $ is  irreducible when $e\equiv  -n$ or $ n\pmod{4n}$,   has two irreducible components (corresponding to $\kappa_*=0$ and $\kappa_*=(n,1)$) when $e\equiv 0\pmod{4n}$, and is empty otherwise.\ 
 \end{itemize}
 This proves (1)(b).
 
  We now assume  that $n$ is prime 
and
prove (1)(c).\ Since $s\mid 2n$, we have $s\in\{1,2,n,2n\}$; the cases $s=1$ and $s=2$ were explained above.\ 
 If $s=n$ (and $n$ is odd), we have $n\mid b$, $n\mid c$, $n\nmid a$, and
  $$  
 e \equiv 4a^2 \pmod{4n}\textnormal{\quad and\quad }\kappa_*=(2a,0). $$ 
  If $s=2n$, the integer $c$ is even, $a$ and $b$ cannot be both even, $n\mid b$, and $n\nmid a$.\ We have
$$ \begin{cases} 
 e \equiv  a^2 \pmod{4n}\textnormal{\quad and\ }\kappa_*=(a,0) &\textnormal{if $2n\mid b$  (hence $a$ is odd);}\\
 e \equiv  a^2+ n \pmod{4n}\textnormal{\quad and\ }\kappa_*=(a,1) &\textnormal{if $b$ is odd (and $n$ is odd);}\\
 e \equiv  a^2+2 \pmod{8}\textnormal{\quad and\ }\kappa_*=(a,1) &\textnormal{if $4\nmid b$ is odd and $n=2$.}
 \end{cases}
 $$  
 
 When $n=2$, one checks that the class of $e$ modulo $8$ (which is in $\{0, 1,2,3, 4,6\}$) completely determines $s$, and $\kappa_*$ up to sign.\ The corresponding divisors $ \cD^{(1)}_{4,2e} $ are therefore all irreducible.
 
When $n\equiv 1\pmod4$, we have $ n\equiv  n^2 \pmod{4n}$ and $ a^2 \equiv  (n-a)^2+ n\pmod{4n}$ when $a$ is odd (in which case $ a^2 \equiv 1\pmod{4}$).\ When  $n\equiv -1\pmod4$, we have $ n(n+1)\equiv 0 \pmod{4n}$ and 
 $ a^2 \equiv  (n-a)^2+ n\pmod{4n}$ when $a$ is even (in which case $ a^2 \equiv 0\pmod{4}$).\ Together with changing $a$ into $-a$ (which does not change the lattice $K$), these are the only coincidences: the corresponding hypersurfaces  $ \cD^{(1)}_{2n,2e} $  therefore have two components and the others are irreducible.\ This proves (1)(c).

\medskip
{\bf Case $\gamma=2$}  (hence $n\equiv -1\pmod4$).\ We may take 
 $h_0:=2\bigl(u+\frac{n+1}{4}v\bigr)+\ell
$, in which case $h_0^\bot=\Z w_1\oplus \Z w_2\oplus M$, with    $w_1:= v+\ell$ and $w_2:=-u+\frac{n+1}{4}v$.\ The matrix of the intersection form on $\Z w_1\oplus \Z w_2$ is    $\left(\begin{smallmatrix}-2&-1\\-1&-\tfrac{n+1}{2}\end{smallmatrix}\right)$ as in \eqref{latns} and the discriminant group  $D(h_0^\bot)\isom \Z/n\Z$ is generated by $(w_1-2w_2)_*=(w_1-2w_2)/n$, with $\bar q((w_1-2w_2)_*)=-2/n$.

Let $(h_0,\kappa')$ be a basis for $K$, so that $\disc(K)=2n\kappa^{\prime 2}-(h_0\cdot \kappa')^2$.\ Since  $\div(h_0)=\gamma=2$, the integer $ h_0\cdot \kappa'$ is even and since $\kappa^{\prime 2}$ is also even (because $\L_{\KKK^{[2]}}$ is an even lattice), we have   $4\mid \disc(K) $ and $-\disc(K)/4$ is a square modulo $n$.\ Since the discriminant of $\L_{\KKK^{[2]}}$ is~2,  the integer $d=|\disc(K^\bot)|$ is either $2\,|\disc(K)|$ or $\tfrac12\,|\disc(K)|$, hence it is even and $e=d/2$ is a square modulo $n$, as desired.

Conversely, it is   easy to construct examples that show that these necessary conditions on $d$ are also sufficient.\ This proves (2)(a).
 
We now prove (2)(b) and (2)(c).\ To prove that the loci $ \cD^{(2)}_{2n,2e} $ are irreducible (when non-empty), we need to show that $ e$ determines $\kappa^2$, and $\kappa_*$ up to sign (where $\kappa$ is a generator of $K\cap h_0^\bot$).

With the notation above, we have $\kappa= ((h_0\cdot \kappa')h_0-2n \kappa')/t$, where $t:=\gcd(h_0\cdot \kappa',2n)$ is even and $\kappa^2=\frac{2n}{t^2}\disc(K)$.\ 
Formula~\eqref{eqd} then gives
\begin{equation*}
2e=\left|\disc(K^\bot)\right|=\left|\frac{\kappa^2\disc(h_0^\bot)}{\div_{h_0^\bot}(\kappa)^2}\right|  = \left|\frac{2n^2\disc(K)}{t^2\div_{h_0^\bot}(\kappa)^2 }\right|.
\end{equation*}
Since $n$ is odd and $t$ is even, and, as we saw above, $\disc(K)\in\{-e,-4e\}$, the only possibility is $\disc(K)=-4e$ and $t\div_{h_0^\bot}(\kappa) =2n$.\ 

Assume that $n$ is square-free and $ n\mid e$.\ Since $-4e=\disc(K)=2n\kappa^{\prime 2}-(h_0\cdot \kappa')^2$, we get $2n\mid (h_0\cdot \kappa')^2$ hence, since $n$ is square-free and odd, $2n\mid h_0\cdot \kappa'$.\ This implies $t=2n$ and $\div_{h_0^\bot}(\kappa)=1$; in particular,   $\kappa_*=0$ and $\kappa^2=-2e/n$ are uniquely determined.\ This proves (2)(b).

We now assume that $n$ is prime.\
 Since $t\div_{h_0^\bot}(\kappa) =2n$ and $t$ is even, 
 \begin{itemize}
 \item either 
$(t,\div_{h_0^\bot}(\kappa),\kappa^2)=(2n,1,-2e/n)$ and  $n\mid  e$; 
\item or
$(t,\div_{h_0^\bot}(\kappa),\kappa^2)=(2,n,-2ne)$ and $n\nmid e$ (because $n\nmid h_0\cdot \kappa'$ and $d=-\frac12 \disc(K)\equiv \frac12(h_0\cdot \kappa')^2\pmod{n}$).
 \end{itemize}
Given $e= a^2+ nn'$, the integer $\kappa^2$ is therefore uniquely determined by $e $: 
\begin{itemize}
\item either $  n\mid a$,  $\kappa^2= -2e/n$, and $\kappa_*=0$;
\item or $n\nmid a$, $\kappa^2= -2ne$, $\kappa_*=\kappa/n$, and $\bar q(\kappa_*)= -2a^2/n\pmod{2\Z}$.
\end{itemize}
In the second case, $\kappa_*=\pm a(w_1-2w_2)_* $; it follows that in all cases, $\kappa_*$ is also uniquely defined, up to sign, by $e$.\ This proves (2)(c).
\end{proof}

\section{The nef cone of a projective hyperk\"ahler fourfold of $\KKK^{[2]}$-type}\label{sect3}

Cones of divisors on projective hyperk\"ahler manifolds of $\KKK^{[n]}$-type were described in~\cite{bht,bama,marsur,mon2}.\
When $n=2$, these results take a very special form.

Let $X$ be a projective hyperk\"ahler fourfold of $\KKK^{[2]}$-type.\ The \emph{positive cone} 
$$\Pos(X)\subset \Pic(X)\otimes\R$$ is the connected component of the open subset $\{ x\in\Pic(X)\otimes\R\mid x^2>0\}$ containing the class of an ample divisor.\ The \emph{movable cone} $$\Mov(X)\subset \Pic(X)\otimes\R$$ is the    (not necessarily open nor closed) convex cone generated by classes of movable  divisors (i.e., those divisors whose  base locus has codimension at least $2$).\  We have  inclusions $\Int(\Mov(X))\subset\Pos(X)$ of the interior of the movable cone into the positive cone, and $\Amp(X)\subset\Mov(X)$ of the ample cone into the movable cone.

We set
\begin{eqnarray*}  
\cD iv_X&:=& \{a\in\Pic(X)\mid a^2=-2\}, \\
\cFlop_X&:=& \{a\in\Pic(X)\mid a^2=-10 ,\  \div_{H^2(X,\Z)}(a)=2\}.
\end{eqnarray*}
Given a divisor class $a\in\Pic(X)\otimes\R$, we denote by $H_a$ the hyperplane
\[
H_a := \{x\in \Pic(X)\otimes\R\mid  x\cdot a=0 \}.
\]

\begin{theo}\label{thm:NefConeHK4}
Let $X$ be a hyperk\"ahler fourfold of $\KKK^{[2]}$-type.

\noindent {\rm (a)} The interior  $\Int(\Mov(X))$ of the movable cone is the connected component of
\[
\Pos(X) \smallsetminus \bigcup_{a\in \cD iv_X} H_a
\]
that contains the class of an ample divisor.

\noindent {\rm (b)} The ample cone $\Amp(X)$ is the connected component of
\[
\Int(\Mov(X)) \smallsetminus \bigcup_{a\in \cFlop_X} H_a
\]
that contains the class of an ample divisor.
\end{theo}

\begin{proof}
Statement (a)  follows from the general result~\cite[Lemma~6.22]{marsur}.\ We sketch instead the proof of (b).

There is an extension $H^2(X,\Z)\subset\widetilde{\Lambda}_X$ of lattices and weight-2 Hodge structures, where the lattice $\widetilde{\Lambda}_X$  is isomorphic to the lattice $U^{\oplus 4}\oplus E_8(-1)^{\oplus 2}$ and the orthogonal $H^2(X,\Z)^\perp \subset \widetilde{\Lambda}_X$ is generated by a primitive vector $v_X$ of square $2$ (\cite[Section 9]{marsur},~\cite[Section~1]{bht}).\ 
 We denote by $\widetilde{\Lambda}_{\textnormal{alg},X}$ the algebraic   (\ie, $(1,1)$-type) part of $\widetilde{\Lambda}_X $, so that $\Pic(X)=v_X^\perp\cap \widetilde{\Lambda}_{\textnormal{alg},X}$.\  Finally, we set
\[
\cFlop_X':= \{\tilde{a}\in \widetilde{\Lambda}_{\textnormal{alg},X}\mid \tilde{a}^2=-2,\  \tilde{a}\cdot v_X=1\}.
\] 
The dual statement   to~\cite[Theorem 1]{bht} is then the following: the ample cone $\Amp(X)$ is the connected component of
\[
\Int(\Mov(X)) \smallsetminus \bigcup_{\tilde{a}\in \cFlop_X'} H_{\tilde{a}}
\]
containing the class of an ample divisor, where the hyperplane $H_{\tilde{a}}$ is defined   as before by $H_{\tilde{a}}:=\{ x\in \Pic(X)\otimes\R \mid x\cdot \tilde{a}=0 \}$.
We notice that the actual statement of~\cite[Theorem 1]{bht} says that we need to exclude the hyperplanes $H_a$, where $a^2\ge -2$ and $|a\cdot v_X|\le 1$.\ 
We may in fact only consider classes with $a^2= -2$, as explained in \cite[Sections 12 and 13]{bama}.

Given a class $\tilde{a} \in \cFlop_X'$, we let $a := 2\tilde{a}-v_X$.\ 
Then $a\in \cFlop_X$ and $H_a=H_{\tilde{a}}$.\ 
Conversely, given $a\in \cFlop_X$, we let $\tilde{b}:=a+v_X\in \widetilde{\Lambda}_{{\textnormal{alg}}}$.\ 
Since $\div_{H^2(X,\Z)}(a)=2$, we have $\tilde{b}=2\tilde{a}$, and $\tilde{a}\in\cFlop_X'$ with $H_{\tilde{a}}=H_a$. This proves (b).
\end{proof}

\begin{rema}\label{rmk:NefConeHK4}
We can make the description in  Theorem~\ref{thm:NefConeHK4} more precise.

\noindent (a) As explained in~\cite[Section 6]{marsur}, it follows from~\cite{mar3} that there is a group of reflections $W_{\Exc}$ acting on $\Pos(X)$.\  Using the Zariski decomposition (\cite{bou}), one shows (\cite[Lemma~6.22]{marsur}) that $W_{\Exc}$ acts faithfully and transitively on the set of connected components of 
\[
\Pos(X) \smallsetminus \bigcup_{a\in \cD iv_X} H_a.
\]
In particular,   $\overline{\Mov(X)} \cap \Pos(X)$ is a fundamental domain for the action of $W_{\Exc}$ on $\Pos(X)$.

\noindent (b) By~\cite[Proposition 2.1]{mat} (see also~\cite[Theorem 7]{hast2}), each connected component of
\[
\Int(\Mov(X)) \smallsetminus \bigcup_{a\in \cFlop_X} H_a
\]
corresponds to the ample cone of a hyperk\"ahler fourfold $X'$ of $\KKK^{[2]}$-type  via   a
  birational map $X\dashrightarrow X'$ which is a composition of Mukai flops with respect to numerically equivalent Lagrangian planes (\cite[Theorem 1.1]{ww}).

\noindent (c) By~\cite[Proposition~4]{add} (generalized to the twisted case in~\cite[proof of Proposition~4.1]{huy}) and~\cite{bama1,bama}, if $\cD iv_X\neq\vide$, the fourfold $X$ is isomorphic to a moduli space $M$ of stable sheaves on a  possibly twisted  K3 surface $(S,\alpha)$.\ The moduli space $M$ is birational to the Hilbert square of a K3 surface if there exists $a\in\cD  iv_X$ with $\div_{H^2(X,\Z)}(a)=2$; otherwise, $M$ is birational to a moduli space of  possibly twisted  rank-2 torsion-free sheaves.

\noindent (d) Similarly, if there exists a  class $w\in \Pic(X)$ with $w^2=0$, the fourfold $X$ is birational to a moduli space $M$ of torsion sheaves on a (possibly twisted) K3 surface $(S,\alpha)$.\
If the divisor class $w$ is also nef and primitive,   $X$ is actually isomorphic to such an $M$  and the Beauville integral system $f\colon X\isomto M\to \P^2$ is a Lagrangian fibration on $X$ such that $w=[f^*\cO_{\P^2}(1)]$.
\end{rema}

Before discussing a few examples of Theorem~\ref{thm:NefConeHK4} when $\Pic(X)$ has rank 2, we briefly review  Pell-type equations (see \cite[Chapter~VI]{nage}).\ Given non-zero integers $e$ and $t$ with $e>0$, we denote by $\cP_e(t)$ the   equation
 \begin{equation}\label{pet}
a^2-eb^2=t,
\end{equation}
where $a$ and $b$ are integers.\ A solution $(a,b)$ of this equation is called positive if $a>0$ and $b>0$.\  If $e$ is not a perfect square, $(a,b)$ is a solution if and only if the norm   $N(a+b\sqrt{e})$ in the quadratic number field $\Q(\sqrt{e})$ is $t$.\  The positive solution with minimal $a$ is called   the minimal solution; it is also the positive solution $(a,b)$ for which the ratio $a/b $ is minimal when $t<0$, maximal when $t>0$.\ 

Assume that $e$ is not a perfect square.\  There is always a minimal solution $(a_1,b_1)$ to the Pell equation $\cP_e(1)$  and if $x_1:=a_1+b_1\sqrt{e}$, all the   solutions of the equation $\cP_e(1)$  correspond to the ``$m$th powers'' $\pm x_1^m$  in $ \Z[\sqrt{e}]$, for $m\in \Z$. 

\begin{exam}[{\cite[Proposition 13.1 and Lemma 13.3]{bama}}\label{exa53}\footnote{Parts of the results of this example were first proved in~\cite[Theorem~22]{hast2} and the rationality of the nef cone was also proved, by very different methods, in~\cite[Corollary 5.2]{ogu}.}]\label{ex:NefConesHilbertSchemes}
Let $(S,L)$ be a  polarized K3 surface such that $\Pic(S)=\Z L$ and $L^2=: 2e$.\ Then
$\Pic( S^{[2]})=\Z L_2\oplus\Z\delta$, where   $L_2$ is the   class on $ S^{[2]}$ induced by $L$ and  
 $2\delta$ is  the class of the divisor in $  S^{[2]} $ that parametrizes non-reduced length-2 subschemes of $S$ (\cite[Remarque, p.~768]{beaa}).\   In the lattice $(H^2(S^{[2]},\Z),q_{S^{[2]}})$, we have the following products
$$L_2^2= 2e\ ,\quad \delta^2=-2\ ,\quad L_2\cdot\delta=0
.$$
 Cones of divisors on $S^{[2]}$ can be described as follows.
\begin{itemize}
\item[(a)] The extremal rays of the (closed) movable cone $\Mov(S^{[2]})$ are spanned by $L_2$ and $L_2-\mu_e\delta$, where
\begin{itemize}
\item[$\bullet$] if $e$ is a perfect square, $\mu_e=\sqrt{e}$;
\item[$\bullet$] if $e$ is not a perfect square and $(a_1,b_1)$ is the minimal solution of the equation $\cP_e(1)$, $\mu_e=e\frac{b_1}{a_1}$.
\end{itemize}
\item[(b)] The extremal rays of the nef cone $\Nef(S^{[2]})$ are spanned by $L_2$ and $L_2-\nu_e\delta$, where
\begin{itemize}
\item[$\bullet$] if the equation $\cP_{4e}(5)$ is not solvable, $\nu_e=\mu_e$;
\item[$\bullet$] if the equation $\cP_{4e}(5)$ is solvable and 
  $(a_5,b_5)$ is its minimal solution, $\nu_e =  2e\frac{b_5}{a_5}$.\footnote{\label{newf}There is a typo in~\cite[Lemma 13.3(b)]{bama}: one should replace $d$ with $2d$.}
\end{itemize}
\end{itemize}
\end{exam} 

\begin{exam}\label{ex:NefConesFano}  
Let $n$ be a positive   integer such that $n\equiv -1\pmod4$.\ Let $(X,H)$ be a polarized hyperk\"ahler fourfold of $\KKK^{[2]}$-type with $H$ of divisibility 2 and   $\Pic(X)=\Z H\oplus \Z L$, with intersection matrix   $\left(\begin{smallmatrix}2n&0\\0&-2e'\end{smallmatrix}\right)$.\ Since any two embeddings of the lattice $K=I_1(2n)\oplus I_1(-2e')$ into $\L_{\KKK^{[2]}}$ for which the image of a generator of $I_1(2n)$ has divisibility 2
differ by an isometry of $\L_{\KKK^{[2]}}$,\footnote{In the notation of the second part of the proof of Proposition~\ref{propirr} (case $\gamma=2$), a generator of $I_1(2n)$ can be sent to the class $h_0$; a generator of $I_1(-2e')$ is then sent  to the class $\kappa'=\kappa$.\ We have $t:=\gcd(h_0\cdot \kappa',2n)=2n$ and the formula $t\div_{h_0^\bot}(\kappa) =2n$  implies $ \div_{h_0^\bot}(\kappa) =1$, \ie, $\kappa_*=0$ in $D(K^\bot)$.\ We then apply Eichler's criterion again in $K^\bot$ and conclude by using the isomorphism $O(\L_{\KKK^{[2]}},h_0)\isom \widetilde O( K^\bot)$.}
they represent very general elements of one component of the special divisor $\cC_{2n,2e'n}^{(2)}$ (we will prove in Theorem~\ref{imper} that they exist if and only if $n >0$ and $e'>1$).\   

{\em We assume in the rest of this example that $n$ is square-free}.\ The hypersurface  $ \cC_{2n,2e'n}^{(2)}$ is then irreducible  by Proposition~\ref{propirr}(2)(b)  and very general elements of $\cC_{2n,2e'n}^{(2)}$ are of the type described above.\ 
 Cones of divisors on $X$ can be described as follows (we set $e:=e'n$).
\begin{itemize}
\item[(a)] The extremal rays of the closure of the  movable cone $\Mov(X)$ are spanned by $H-\mu_{n,e} L$ and $H+\mu_{n,e} L$, where
\begin{itemize}
\item[$\bullet$] if the equation $\cP_{e}(-n)$ is not solvable, $\mu_{n,e}=n/\sqrt{e}$;
\item[$\bullet$] if the equation $\cP_{e}(-n)$ is solvable and $(a_{-n},b_{-n})$ is its minimal solution, $\mu_{n,e}=\frac{a_{-n}}{e'b_{-n}}$.
\end{itemize}
\item[(b)] The extremal rays of the nef cone $\Nef(X)$ are spanned by $H-\nu_{n,e} L$ and $H+\nu_{n,e} L$, where
\begin{itemize}
\item[$\bullet$] if the equation $\cP_{4e}(-5n)$ is not solvable, $\nu_{n,e}=\mu_{n,e}$;
\item[$\bullet$] if the equation $\cP_{4e}(-5n)$ is solvable and $(a_{-5n},b_{-5n})$ is its minimal solution, $\nu_{n,e}=\frac{a_{-5n}}{2e'b_{-5n}}$.
\end{itemize}
\end{itemize}

To prove these statements, it is enough to notice that, in the notation of Theorem~\ref{thm:NefConeHK4}, a class in $\cD iv_X$ corresponds to a solution to the equation $\cP_{e}(-n)$; similarly, a class in $\cFlop_X$ corresponds to a solution to the equation $\cP_{4e}(-5n)$.\ The description of the cones of divisors on $X$ then follows   from Theorem~\ref{thm:NefConeHK4} by a direct computation.
\end{exam}

\section{The image of the period map}\label{sec8}

The description  of the cones of divisors for hyperk\"ahler fourfolds of $\KKK^{[2]}$-type given in  Section~\ref{sect3} easily implies our main result on the images of their period maps.
 
  \begin{theo}\label{imper}
Let $n$ be a positive integer and let $\gamma\in\{1,2\}$.\ The   image of the period map
  \begin{equation*} 
 \wp^{(\gamma)}_{2n}\colon\cM_{2n}^{(\gamma)}\lra  \cP_{2n}^{(\gamma)}
 \end{equation*}
  is exactly the complement of the union of finitely many Heegner divisors.\ More precisely, these Heegner divisors are
\begin{itemize}
\item if $\gamma=1$, 
\begin{itemize}
\item  some irreducible components of the hypersurface $\cD^{(1)}_{2n,2n}$ (two components if $n\equiv 0$ or $1\pmod4$, one component if $n\equiv 2$ or $3\pmod4$);
\item one irreducible component of the hypersurface $\cD^{(1)}_{2n,8n}$;
\item one irreducible component of the hypersurface $\cD^{(1)}_{2n,10n}$;
\item and, if $ n \equiv  \pm 5\pmod{25}$, some irreducible components of the hypersurface $\cD^{(1)}_{2n,2n/5}$;
\end{itemize}
\item if $\gamma=2$ (and $n\equiv   -1\pmod4$), one irreducible component of the hypersurface $\cD^{(2)}_{2n,2n}$.
\end{itemize}
\end{theo}

\begin{rema}\label{qqq}
Assume that $n$ is square-free (so in particular $n\not\equiv 0\pmod4$).\ We proved in Proposition~\ref{propirr} that   
\begin{itemize}
\item the hypersurface $\cD^{(1)}_{2n,2n}$ has two components if $n\equiv  1\pmod4$, one component otherwise;
\item the hypersurface $\cD^{(1)}_{2n,8n}$ has two components if $n\equiv  -1\pmod4$, one component otherwise;
\item the hypersurface $\cD^{(1)}_{2n,10n}$ has two components if $n\equiv   1\pmod4$, one component otherwise;
\item the hypersurface $\cD^{(1)}_{2n,2n/5}$ has two components if $n\equiv   1\pmod4$, one component otherwise;
\item the hypersurface $\cD^{(2)}_{2n,2n}$ is irreducible (when $n\equiv   -1\pmod4$).
\end{itemize}
Furthermore, it follows from the proof of the theorem that when moreover $ n \equiv  \pm 5\pmod{25}$,
\begin{itemize}
\item the hypersurface $\cD^{(1)}_{2n,2n/5}$ is irreducible.
\end{itemize}
\end{rema}

\begin{proof}[Proof of Theorem~\ref{imper}]
Take a  point $x\in \cP_{2n}^{(\gamma)}$.\ Since the period map for smooth compact (not necessarily projective) hyperk\"ahler fourfolds is surjective (\cite[Theorem~8.1]{huy1}), there exists a compact hyperk\"ahler fourfold $X'$ with the given period point $x$.\  Since the class $h_0$ is algebraic and has positive square, $X'$ is projective by~\cite[Theorem~3.11]{huy1}.\ Moreover, the class $h_0$ corresponds to the class of an integral divisor $H$ in the positive cone of $X'$.\ By Remark~\ref{rmk:NefConeHK4}(a), we can let an element in the group $W_{\Exc}$ act and assume that the pair $(X',H)$, representing the period point $x$ and the class $h_0$, is such that $H$ is in $\overline{\Mov(X')}\cap\Pos(X')$.\  By Remark~\ref{rmk:NefConeHK4}(b), we can find a projective hyperk\"ahler fourfold $X$  which is birational to $X'$ (hence still has period~$x$), such that the divisor $H$, with class $h_0$, is nef and big on $X$ and has divisibility $\gamma$.\ 
Note that, since $X'$ is birational to $X$, it is deformation equivalent to $X$ (\cite[Theorem~4.6]{huy1}), hence still of $\KKK^{[2]}$-type.
 
To summarize, the point $x$ is in the image of the period map $\wp^{(\gamma)}_{2n}$ if and only if $H$ is actually ample on $X$.\ We now  use Theorem~\ref{thm:NefConeHK4}: $H$ is ample if and only if it is not orthogonal to any algebraic class either with square $-2$, or with square $-10$ and divisibility~2. 

{\em If $H$ is orthogonal to an algebraic class $\kappa$ with square $-2$,} the Picard group of $X$ contains a rank-2 lattice $K$ with intersection matrix  $\bigl(\begin{smallmatrix}2n&0\\0&-2\end{smallmatrix}\bigr)$; the fourfold $X$ is therefore special of discriminant $2e:=-\disc(K^\perp)$ (its period point is in the hypersurface $\cD^{(\gamma)}_{2n,K} $).\ 

If $\gamma=1$,   the   divisibility $s:=\div_{h_0^\perp}(\kappa)$ is either 1 or 2.\ By \eqref{kap}, we have   $es^2=-2n\kappa^2=4n$, hence
\begin{itemize}
\item either $s=1$, $e=4n$, and $\kappa_*=0$: the period point is then in one irreducible component of the hypersurface $\cD^{(1)}_{2n,8n}$;
\item or $s=2$, $e= n$, and 
\begin{itemize}
\item either $\kappa_*=(0,1)$;
\item or $\kappa_*=(n,0)$ and $n\equiv 1\pmod4$;
\item or $\kappa_*=(n,1)$ and $n\equiv 0\pmod4$.
\end{itemize}
\end{itemize}
The period point $x$ is  in one irreducible component of the hypersurface $\cD^{(1)}_{2n,2n}$ if $n\equiv 2$ or $3\pmod4$, or in the union of two such components otherwise.

If $\gamma=2$, we have $e=-\disc(K)/4=n$, $t =\sqrt{2n\disc(K)/\kappa^2}=2n $, and $\div(\kappa)=2n/t=1$, hence $\kappa_*=0$: the period point $x$ is in one irreducible component of the hypersurface $\cD^{(2)}_{2n,2n}$.

{\em If $H$ is orthogonal to an algebraic class with square $-10$ and divisibility 2,} the Picard group of $X$ contains a rank-2 lattice $K$ with intersection matrix  $\bigl(\begin{smallmatrix}2n&0\\0&-10\end{smallmatrix}\bigr)$, hence $X$ is   special of discriminant $2e:=-\disc(K^\perp)$.\ Again, we distinguish two cases, keeping the same notation.

If $\gamma=1$,   the   divisibility $s:=\div_{K^\perp}(\kappa)$ is even (because the divisibility in $H^2(X,\Z)$ is 2) and divides $\kappa^2=-10$, hence it is
 either  2 or 10 .\ Moreover, $es^2=-2n\kappa^2=20n$, hence
\begin{itemize}
\item either  $s=2$, $e= 5n$, and $\kappa_*=(0,1)$: the period point $x$ is then in one irreducible component of the hypersurface $\cD^{(1)}_{2n,10n}$;
\item or $s=10$ and $e= n/5$: the period point is then in   the hypersurface $\cD^{(1)}_{2n,n/5}$.
\end{itemize}
In the second case, since the divisibility of $\kappa$ in $H^2(X,\Z)$ is 2,
 $a$ and $c$ are even, so that $b$ is odd and $\kappa_*=(a,1)$.\ We have $10=s=\gcd(2na,2b,2c)$, hence $b$ and $c$ are divisible by $5$, but not $a$, because $\gcd( a, b, c)=1$.\ We have $e \equiv  a^2+ n \pmod{4n}$, hence $ e\equiv    a^2\equiv  \pm1\pmod{5}$.\

In general, there are many possibilities for $a=2a'$, with $a^{\prime 2}\equiv e\pmod {5e}$.\ However, if $n$ is square-free, $e$ divides $a'$ and   $  (a'/e)^2\equiv 1 \pmod{5}$, so that $a \equiv  \pm 2e\pmod{2n}$.\ It follows that
 $\pm a$ (hence also $\pm \kappa_*$) is well determined (modulo $2n$), so we have a single component of $\cD^{(1)}_{2n,n/5}$.

If $\gamma=2$, we have $e=-\disc(K)/4=5n$ and $t^2 = {2n\disc(K)/\kappa^2}=n^2/10 $, which is impossible.

Conversely, in each case  described above, it is easy to construct a class $\kappa$ with the required square and divisibility which is orthogonal to $H$.
\end{proof}

\begin{exam}[Double EPW sextics: $n=\gamma=1$]\label{qqq1}
Double EPW sextics  were defined in~\cite{og4} as ramified  double covers of certain singular sextic  hypersurfaces in $ \P^5$.\ When smooth, they are hyperk\"ahler fourfolds of $\KKK^{[2]}$-type with a polarization of degree 2.\ They fill out a dense open subset $\cU_2^{(1)}$ of $\cM_2^{(1)}$ whose complement contains the irreducible hypersurface $\cH_2^{(1)}$ whose general points correspond to     pairs $(S^{[2]}, L_2-\delta)$, where~$(S,L)$ is a polarized K3 surface of degree 4 (\cite[Section 5.3]{og6}).   
 
  O'Grady proved that the image of $\cU_2^{(1)}$ in the period space does not meet  $\cD_{2,2}^{(1)} $, $\cD_{2,4}^{(1)} $,  $\cD_{2,8}^{(1)} $, and one component of $\cD_{2,10}^{(1)} $ (\cite[Theorem~1.3]{og6}\footnote{O'Grady's hypersurfaces $\mathbb{S}'_2\cup \mathbb{S}''_2$,    $\mathbb{S}_4$,
$\mathbb{S}^\star_2$, are our $\cD_{2,2}^{(1)} $, $\cD_{2,4}^{(1)} $,  $\cD_{2,8}^{(1)} $.}); moreover, by~\cite[Theorem~8.1]{dims},  this image does meet   all the other components of the  non-empty hypersurfaces  $\cD_{2,d}^{(1)} $.\  The hypersurface $\cH_2^{(1)}$ maps to $\cD_{2,4}^{(1)} $.\ These results agree with Theorem~\ref{imper} and  Remark~\ref{qqq}, which say that the image of $\cM_2^{(1)}$ in the period space is the complement of the union of $\cD^{(1)}_{2 ,2 }$, $\cD^{(1)}_{2 ,8}$, and one of the two components of $\cD^{(1)}_{2 ,10}$.\ However, our theorem says nothing about the image of 
 $\cU_2^{(1)}$.\ O'Grady conjectures that it is the complement of the   hypersurfaces $\cD_{2,2}^{(1)} $, $\cD_{2,4}^{(1)} $,  $\cD_{2,8}^{(1)} $, and one component of $\cD_{2,10}^{(1)} $; this would follow if one could prove $\cM_2^{(1)}=\cU_2^{(1)}\cup \cH_2^{(1)}$.
 \end{exam}

\begin{exam}[Varieties of lines on cubic fourfolds: $n=3$ and $\gamma=2$]
If $W\subset \P^5$  is a smooth cubic fourfold, the variety $F(W)$ of lines contained in $W$ is a hyperk\"ahler fourfold and its Pl\"ucker polarization has square $ 6$ and divisibility~2 (\cite{bedo},~\cite[Proposition 2.1.2]{has}).\ These fourfolds fill out a dense open subset  $\cU_6^{(2)}$  of $\cM_6^{(2)}$ whose complement contains an irreducible hypersurface $\cH_6^{(2)}$ whose general points correspond to   pairs $(S^{[2]},2L_2-\delta)$, where~$(S,L)$ is a polarized K3 surface of degree 2 (see Proposition~\ref{coro43}). 

Theorem~\ref{imper} and Remark~\ref{qqq} say that the image of $\cM_6^{(2)}$ in the period space is the complement of the irreducible hypersurface $\cD^{(2)}_{6 ,6}$.\ This (and much more) was first proved by Laza in~\cite[Theorem~1.1]{laza}, together with the fact that $\cM_6^{(2)}=\cU_6^{(2)}\cup \cH_6^{(2)}$; since $\cH_6^{(2)} $ maps onto $\cD^{(2)}_{6 ,2}$,  the image of $\cU_6^{(2)}$  is the complement of $\cD^{(2)}_{6 ,2}\cup \cD^{(2)}_{6 ,6}$.
\end{exam}

\section{Unexpected isomorphisms between hyperk\"ahler fourfolds}\label{sec9}

In this section, we study birational isomorphisms between components of various Noether--Lefschetz loci induced by ``unexpected'' isomorphisms between hyperk\"ahler fourfolds.\  We treat first the case of Hilbert squares.

\subsection{Special hyperk\"ahler fourfolds isomorphic to Hilbert squares of   K3 surfaces}\label{sec71}
If a polarized hyperk\"ahler fourfold is isomorphic to the Hilbert square of a  K3 surface, it is special in the sense  defined in Section~\ref{sec22}.\  We use  standard notation for cohomology classes on a Hilbert
square (see Example~\ref{ex:NefConesHilbertSchemes}).\ 
The slope $\nu_e$ was defined in the same example and the special loci $\cC_{2n,2e}^{(\gamma)}\subset \cM_{2n}^{(\gamma)}$ in Section~\ref{sec22}.

\begin{prop} \label{thh}
Let $n$ and $e$  be positive integers.\ 
Assume that the equation $\cP_{e}(-n)$ (see \eqref{pet}) has a positive solution $(a,b)$ that satisfies the conditions
\begin{equation}\label{nue}
\frac{a}{b}<\nu_e\qquad{and}\qquad \gcd(a,b)=1.
\end{equation}
If $\cK_{2e}$ is the moduli space  of polarized K3 surfaces of degree $2e$, the rational map
\begin{eqnarray*}
\varpi\colon \cK_{2e} &\dra& \cM_{2n}^{(\gamma)}\\
(S,L)&\longmapsto&(S^{[2]}, bL_2-a\delta ),
\end{eqnarray*}
where $\gamma=2$ if $b$ is even, and $\gamma=1$ if $b$ is odd,
induces a birational isomorphism onto an irreducible component of $\cC_{2n,2e}^{(\gamma)}$.
 \end{prop}
 
\begin{proof}
If $(S,L)$ is a   polarized K3 surface of degree $2e$ and $K:=\Z L_2\oplus \Z\delta \subset H^2(S^{[2]},\Z)$, the lattice $K^\bot$ is the orthogonal in $H^2(S,\Z)$ of the class $L$.\ Since the lattice $H^2(S,\Z)$ is unimodular, $K^\bot$ has discriminant $-2e$, hence $S^{[2]}$ is special of discriminant  $2e$.

The class $H:=bL_2-a\delta$ has   divisibility~$\gamma$ and square $2n$.\ It is   primitive, because $\gcd(a,b)=1$,  and ample on $S^{[2]}$ when $\Pic(S)=\Z L$  because of   the inequality in~\eqref{nue}.\ 
Therefore, the pair $(S^{[2]},H)$ corresponds to a point of $\cC_{2n,2e}^{(\gamma)}$.
 
The map $\varpi$ therefore sends   a very general point  of $\cK_{2e}$ to $\cC_{2n,2e}^{(\gamma)}$.\ To prove   that $\varpi$ is generically injective, we assume to the contrary that there is an isomorphism $\phi\colon S^{[2]}\isomto S^{\prime [2]}$ such that $\phi^*(bL'_2-a\delta')=bL_2-a\delta$, although  $(S,L)$ and $(S',L')$ are not isomorphic.\ It is straightforward to check that this implies $\phi^*\delta'\ne \delta$ and  that the extremal rays of the nef cone of $S^{[2]}$ are spanned by the primitive classes $L_2$ and $\phi^*L'_2$.\ Comparing this with the description of the nef cone given in Example~\ref{exa53}, we see that $e$ is not a perfect square, $\phi^*L'_2=a_1L_2-eb_1\delta$ and $\phi^*(a_1L'_2-eb_1\delta')=L_2$, where 
 $(a_1,b_1)$ is the minimal solution to the Pell equation $\cP_e(1)$.\ The same proof as that of Theorem~\ref{bbb} implies  $e>1$, the equation $\cP_{e}(-1)$ is solvable  and  the equation $\cP_{4e}(5)$ is not.

By Theorem~\ref{bbb} again, $S^{[2]}$ has a non-trivial involution $\sigma$ and $(\phi\circ\sigma)^*(L'_2)=L_2$ and $(\phi\circ\sigma)^*(\delta')=\delta$.\ This implies that $\phi\circ\sigma$ is induced by an isomorphism $(S,L) \isomto (S',L')$, which contradicts our hypothesis.\  The map $\varpi$ is therefore generically injective and since $\cK_{2e}$ is irreducible of  dimension~19, its image is a component of $\cC_{2n,2e}^{(\gamma)}$.
\end{proof}

\begin{rema}\label{ex:nprimebeven}
Assume that $n$ is prime.\ The locus $\cC_{2n,2e}^{(2)}$ is irreducible by Proposition~\ref{propirr}.\  Therefore, under the assumptions of Proposition~\ref{thh} and when  $b$ is even, we have a birational isomorphism $\cK_{2e}\isomdra\cC_{2n,2e}^{(2)}$.\ When $e>61$, the varieties $\cK_{2e}$ are known to be of general type   (\cite{ghs0}), hence  so is $\cC_{2n,2e}^{(2)}$.\ More precise results on the geometry of the varieties $\cC_{6,2e}^{(2)}$ can be found in~\cite{nue,vas,lai}.
\end{rema}

\begin{exam}\label{ex:HilbEPW}
Assume $n=1$.\  Under the assumptions of Proposition~\ref{thh},  $b$ is odd.\ The locus $\cC_{2,2e}^{(1)}$   has either one or two components, according to whether $e$ is even or odd (Proposition~\ref{propirr}).\ If $e$ is odd, we have $e>1$ and one checks that the image of $\varpi$ is the component of $\cC_{2,2e}^{(1)}$ denoted by   $\cC_{2,2e}^{(1)\prime\prime}$ in~\cite[Section~4.3]{og6}.\
Therefore, we have  birational isomorphisms
\begin{eqnarray*}
\varpi\colon \cK_{2e} &\isomdra& \begin{cases}\cC_{2,2e}^{(1)}&\textnormal{ if $e$ is even;}  \\
\cC_{2,2e}^{(1)\prime\prime}&\textnormal{ if $e$ is odd.} 
\end{cases}
\end{eqnarray*}
\end{exam}

\begin{rema}\label{pesq}
  When  $e $ is a perfect square, the  positive solutions $(a,b)$ to the equation  $\cP_{e}(-n)$
satisfy  $a- b\sqrt{e}= -n''$ and $a+ b\sqrt{e}=n' $, with $n=n'n''$.\ This implies $a=\frac12(n'-n'')$ and $b\sqrt{e}=\frac12(n'+n'')$, hence $0<n''<n'$.\ We then have $\frac{a}{b}=\frac{n'-n''}{n'+n''}\sqrt{e}<\sqrt{e}=\nu_e$ hence  Proposition~\ref{thh}  applies to all positive solutions $(a,b)$ of the equation $ \cP_{e}(-n)$  with $\gcd(a,b)=1$.\ In particular, when $n$ is odd and $n>1$, we obtain a geometric description of the fourfolds corresponding to general points of some component of $\cC_{2n,2}^{(\gamma)}$, where $\gamma=1$ if $n\equiv 1\pmod4$, and $\gamma=2$ if $n\equiv -1\pmod4$ (take $n'=n$ and $n''=1$).\end{rema}

\begin{rema}\label{rmk:Picard2IsomorphicHilb2}
Under the hypotheses of   Proposition~\ref{thh}, one can show that \emph{all} polarized hyperk\"ahler fourfolds $(X,H)$ with Picard number $2$ which are in the component of $\cC_{2n,2e}^{(\gamma)}$ dominated by $\cK_{2e}$ are actually isomorphic to a Hilbert square $S^{[2]}$; however, {\em some} generality condition on $X$ is needed:  the varieties  of lines of some smooth cubic fourfolds of discriminant~$14$ ($n=3$, $\gamma=2$, $e=7$) are not isomorphic to the Hilbert square of a K3 surface.
\end{rema}

We deduce from Proposition~\ref{thh} a characterization of Hilbert squares of general polarized K3 surfaces that are isomorphic to double EPW sextics.

 \begin{coro}\label{cor53}
Let $e$ be an integer such that $e\ge3$ and let $(S,L )$ be a general polarized K3 surface of  degree $2e$.\ The following conditions are equivalent:
\begin{itemize}
\item [{\rm (i)}] the equation $\cP_{e}(-1)$ is solvable and the equation $\cP_{4e}(5)$ is not;
\item [{\rm (ii)}] the equation $\cP_{e}(-1)$ has a positive solution $(a,b)$ such that  $\frac{a}{b}<\nu_e$;
\item [{\rm (iii)}]   the Hilbert square    $S^{[2]}$ is isomorphic   to  a double EPW sextic  of discriminant $2e$;
\item [{\rm (iv)}] the variety $S^{[2]}$ has a non-trivial automorphism.
\end{itemize}
When these conditions are realized, $S^{[2]}$ then has a non-trivial involution $\sigma$, the  quotient  $S^{[2]}/\sigma$  is an EPW sextic $Y\subset \P^5$, and  the complete linear system $|bL_2-a\delta|$ defines a morphism which  factors as 
 $S^{[2]}\thra S^{[2]}/\sigma= Y \hra \P^5$. 
\end{coro}

\begin{proof}
The equivalence   (i) $\Leftrightarrow$ (iv) is Theorem~\ref{bbb}.\
 The implication (iv) $\Rightarrow$ (ii) comes from the facts that the equation $\cP_e(-1)$ has a minimal solution $(a_{-1},b_{-1})$ and, if $\sigma$ is the non-trivial automorphism of $S^{[2]}$ (Theorem~\ref{bbb}),  the class $b_{-1}L_2-a_{-1}\delta$ is positively proportional to $ L_2+\sigma^*L_2 $, hence  ample.\
 The implication (ii) $\Rightarrow$ (iii) is Proposition~\ref{thh} and Example~\ref{qqq1}.\ The implication (iii) $\Rightarrow$ (iv) is obvious.\  The consequences stated at the end follow from~\cite[Section~4]{og7}, which explains why $\dim(|H|)=5$, where $H$ is the canonical polarization on  
the double EPW sextic.\end{proof}

\begin{rema}\label{rema54}  When $e=2$, all the conditions of Corollary~\ref{cor53} hold  except for (iii).\ The fourfold   $S^{[2]}$ carries the non-trivial Beauville involution $\sigma$ (Example~\ref{sma})  and the  complete linear system $| L_2- \delta|$ defines a morphism which  factors as 
 $S^{[2]}\thra S^{[2]}/\sigma\stackrel{3:1}{\thra} \Gr(2,4)\hra \P^5$.\  This        fits with the fact that $3\Gr(2,4)$ is a (degenerate) EPW sextic (\cite[Claim~2.14]{og3}).  \end{rema} 
 
\begin{exam}\label{rema54b} When $e=13$,   the equivalent conditions of Corollary~\ref{cor53} are satisfied, hence  the Hilbert square  of a general polarized K3 surface $(S,L)$ of degree 26 is a double EPW sextic, with canonical involution $\sigma$.\ Moreover,
  two positive solutions, $ (7,2)$ and $(137,38)$, of  the equation $\cP_{  13}(-3)$ satisfy the conditions \eqref{nue} of Proposition~\ref{thh} with $b$ even.\  It follows that $S^{[2]}$ is also isomorphic to a general element of $\cC_{6,26}^{(2)}$, \ie, to the variety of lines $F(W)$ on a special cubic hypersurface $W\subset \P^5$ of discriminant $26$ (the two isomorphisms $S^{[2]}\isom F(W)$ differ by $\sigma$, and $\sigma^*(2L_2-7\delta)=38L_2-137\delta$).
  \end{exam}

We now show that given any positive integer $n$, Proposition~\ref{thh} applies to infinitely many integers $e$.

\begin{prop}\label{coro43}
Let $n$ be a positive integer.\ There are  infinitely many distinct hypersurfaces in the moduli spaces $\cM_{2n}^{(1)}$, and $\cM_{2n}^{(2)}$ if $n\equiv -1\pmod4$, whose general points correspond to Hilbert squares of  K3 surfaces.\ In both cases, the union of these hypersurfaces is dense in the moduli space for the euclidean topology.
\end{prop}

\begin{proof}[Sketch of proof]
When $m>0$, the pair $(m,1)$ is a solution of the equation $\cP_{e}(-n)$, with $e=m^2+n$, and one easily checks   that the inequality  $m<\nu_e$ holds when $(n,m)\ne (1,2)$.
 
When $m\ge 0$, the pair $(2m+1,2)$ is a solution of the equation $\cP_{e}(-n)$, with $e=m^2+m+\tfrac{n+1}{4}$ and   one easily checks that the inequality $m+\tfrac12<\nu_e $ holds when $(n,m)\ne (3,1)$.
 
Finally, the density statement follows from a powerful result of Clozel and Ullmo (Theorem~\ref{thcu} below).
 \end{proof}
 
 \begin{theo}[Clozel--Ullmo]\label{thcu}
The union of infinitely many Heegner divisors in any moduli space $\cM^{(\gamma)}_{2n}$ is dense for the euclidean topology. 
\end{theo}

\begin{proof}
This follows from the main result of~\cite{clul}: the space $\cM^{(\gamma)}_{2n}$  is a  Shimura variety and   each Heegner divisor $\cD_x$ is a ``strongly special'' subvariety, hence is endowed with a canonical probability measure $\mu_{\cD_x}$.\ Given any infinite family $(\cD_{x_a})_{a\in\N}$ of Heegner divisors, there exists a subsequence $(a_k)_{k\in\N}$, a  strongly special subvariety $Z \subset \cM^{(\gamma)}_{2n}$ which contains  $\cD_{x_{a_k}}$ for all $k\gg 0$ such that $(\mu_{\cD_{x_{a_k}}})_{k\in\N}$ converges weakly to $\mu_Z$ (\cite[th.~1.2]{clul}).\ For dimensional reasons, we have $Z= \cM^{(\gamma)}_{2n}$; this implies that $\bigcup_a \cD_{x_a}$ is dense in    
$ \cM^{(\gamma)}_{2n}$.
\end{proof}

\begin{rema}
It was proved in~\cite{mm} that Hilbert schemes of projective K3 surfaces are dense in the coarse moduli space of all (possibly non-algebraic) \hKm s  of $\KKK^{[n]}$-type.
\end{rema}

\subsection{Isomorphisms between various special hyperk\"ahler fourfolds}\label{sec72}
We now apply a similar construction with the polarized hyperk\"ahler fourfolds $(X,H)$ studied in Example~\ref{ex:NefConesFano}, whose notation we keep.\ For the sake of simplicity, we assume that $n$ is square-free; these fourfolds then correspond  to points of the irreducible hypersurface $\cC_{2n,2e'n}^{(2)}$.

\begin{prop} \label{thhw}
Let $n$, $m$, and $e$ be positive integers.\ Assume that $n$ is square-free, $n\equiv -1\pmod4$,   $n\mid e$, and $n\ne e$.\
Assume further that the equation $\cP_{e}(nm)$ has a solution $(na,b)$ with $a>0$ that satisfies the conditions
\begin{equation}\label{nuep}
\frac{|b|}{a}<\nu_{n,e}\qquad and \qquad \gcd(a,b)=1.
\end{equation}
There is a  rational map
\begin{eqnarray*}
\varpi\colon \cC_{2n,2e}^{(2)} &\dra& \cM_{2m}^{(\gamma)}\\
(X,H)&\longmapsto&(X,aH+bL),
\end{eqnarray*}
where $\gamma=2$ if $b$ is even, and $\gamma=1$ if $b$ is odd.\ This map  
induces a birational isomorphism onto an irreducible component of $\cC_{2m,2e}^{(\gamma)}$.
\end{prop}

In the proposition, the locus $\cC_{2n,2e}^{(2)}$ is non-empty and irreducible by Proposition~\ref{propirr} and Theorem~\ref{imper}.

\begin{proof}
The proof is the same as that of  Proposition~\ref{thh} and is based on the fact that if $(X,H)$ corresponds to  a very general point of $\cC_{2n,2e}^{(2)}$, the class $ aH+bL$ is primitive, has square $2a^2n-2b^2e'=2m$ and divisibility~$\gamma$, and is  ample on $X$ because of  the inequality in~\eqref{nuep}.\  Therefore, the pair $(X,aH+bL)$ corresponds to a point of $\cC_{2m,2e}^{(\gamma)}$.

To prove   that $\varpi$ is generically injective,   assume that there is an isomorphism $\phi\colon X\isomto X'$ such that $\phi^*(aH'+bL')=aH+bL$.\ If $\phi^*H'\ne H$,   the matrix of $\phi^*$ in the bases $(H',L')$ and $(H,L)$ is that of a non-trivial isometry with a fixed vector, hence a reflection.\  

As we will see during the proof of Proposition~\ref{autfw*}, the matrix of such an isometry that extends to an isometry between $H^2(X',\Z)$ and $H^2(X,\Z)$ must be of the form $\bigl(\begin{smallmatrix} 2s^2e'+1&-2e'rs\\2nrs&-(2s^2e'+1) \end{smallmatrix}\bigr)$, where  $(nr,s)$ is a solution to the equation $\cP_e(n)$, both equations $\cP_{e}(-n)$ and $\cP_{4e}(-5n)$ are not solvable, and $e$ is not a perfect square.\  We then have $\phi^*(rH'+sL')= rH+sL$.\ Since $\phi^*(aH'+bL')=aH+bL$ and $ aH'+bL' $ is primitive, we must have $m=1$ (and $a=r$, $b=s$).\ In that case, by Proposition~\ref{autfw*}, $X$ does have an involution $\sigma$ that acts as on $\Pic(X)$ as the reflection with axis spanned by $rH+sL$.\ The isomorphism $  \phi\circ \sigma\colon   X\isomto X'$ then pulls back $H'$ to $H$.\ This proves the proposition.
  \end{proof}
  
\begin{exam}\label{ex:mprimebeven}
Under the assumptions of Proposition~\ref{thhw}, when $m$ is prime and $b$ is even, the locus $\cC_{2m,2e}^{(2)}$ is irreducible by Proposition~\ref{propirr}.\  Therefore, there is a birational isomorphism $\cC_{2n,2e}^{(2)}\isomdra\cC_{2m,2e}^{(2)}$.
\end{exam}

\begin{exam}
Assume $m=1$.\ As in Example~\ref{ex:HilbEPW}, we have, under the assumptions of Proposition~\ref{thhw}, birational isomorphisms
\begin{eqnarray*}
\varpi\colon \cC_{2n,2e}^{(2)} &\isomdra& \begin{cases}\cC_{2,2e}^{(1)}&\textnormal{ if $e$ is even;}  \\
\cC_{2,2e}^{(1)\prime\prime}&\textnormal{ if $e$ is odd.} 
\end{cases}
\end{eqnarray*}
\end{exam}

\begin{rema}\label{invo}  
Given a pair $(a,b)$ that satisfies the conditions~\eqref{nuep}, we can construct two maps $\varpi^\pm$ by sending $(X,H)$  either to $(X,aH+bL)$ or to $(X,aH-bL)$.\ These two maps are distinct unless there exists an  automorphism $\phi$ of $X$ that sends $aH+bL$ to $aH-bL$.\ One checks using the computations of the proof of Proposition~\ref{autfw*} that this is only possible when we are in case (a) of that proposition,  $\phi^*$   acts as a rotation $\bigl(\begin{smallmatrix} 2s^2e+1& 2e'rs\\2nrs& 2s^2e +1  \end{smallmatrix}\bigr)$ on $\Pic(X)$, with $r^2-es^2=1$, and moreover, $a=r$, $b=ns$, and $m=n$.\ The maps $\varpi^\pm\colon \cC_{2n,2e}^{(2)}\isomlra\cC_{2n,2e}^{(2)}$ then correspond to changing the polarization by an automorphism of $X$: they are just particular cases of an infinite family of such maps.
 \end{rema}

As in Section~\ref{sec72}, we characterize which of our special hyperk\"ahler fourfolds   are  isomorphic to double EPW sextics.

\begin{coro}\label{thhepw}
Let $n$ and $e$ be positive integers.\ Assume that $n$ is square-free, $n\equiv -1\pmod4$,   $n\mid e$, and $n\ne e$.\ Let $(X,H)$ be a polarized hyperk\"ahler fourfold corresponding to a general point of $\cC_{2n,2e}^{(2)}$.\  Then $X$ is isomorphic to a double EPW sextic if and only if the equation $\cP_{e}(n)$ is solvable but the equation $\cP_{4e}(-5n)$ is not.
\end{coro}
  
Under the hypotheses of the corollary, the automorphism group of $X$ is isomorphic to   $\Z\rtimes \Z/2\Z$  hence contains infinitely many involutions $(\sigma_m)_{m\in\Z}$ (Proposition~\ref{autfw*}).\ When $X$ is very general, all the quotients $X/\sigma_m$ are EPW sextics.

\begin{proof}
We may assume that $(X,H)$ is very general in $\cC_{2n,2e}^{(2)}$.\ If $X$ is isomorphic to a double EPW sextic, it has a non-trivial automorphism  and the conclusion follows from Proposition~\ref{autfw*}.\ Conversely, if the equation $\cP_{e}(n)$ is solvable but the equation $\cP_{4e}(-5n)$ is not, one checks that $e$ is not a perfect square, hence $X$ has, by Proposition~\ref{autfw*}, a non-trivial involution $\sigma$ (and in fact, countably many such involutions) that fixes a square-2 class $ rH+sL $ which is positively proportional to $ H +\nu_{n,e}L +\sigma^*(H +\nu_{n,e}L) $, hence ample.\ By Proposition~\ref{thhw}, the pair $(X,rH+sL)$ is a general element of $\cC_{2 ,2e}^{(1)}$, hence $X$ is a double EPW sextic by Example~\ref{qqq1} (note that $e\ge 2n\ge 6$). \end{proof}

\begin{rema}\label{ex:EPW}
Assume that both equations $\cP_{4e}(-5n)$ and $\cP_{e}(n)$ are solvable.\  As in Remark~\ref{rmk:TwoBirationalModels}, let $X'$ be the other birational model of a general $X$ in an irreducible component of $\cC_{2n,2e}^{(2)}$.\  Then $X'$ is isomorphic to a double EPW sextic by the same proof as above.
\end{rema}

Finally, we   show that given any positive integer $n$, Proposition~\ref{thhw} applies to infinitely many integers $e$.

\begin{coro}\label{coro521}
Let $n$ be a positive square-free integer such that $n\equiv -1 \pmod 4$.\  There are  infinitely many distinct hypersurfaces in the moduli space $\cM_{2n}^{(2)}$ whose general points correspond to double EPW sextics.\  Their union is dense in $  \cM_{2n}^{(2)}$.
\end{coro}
 
\begin{proof} 
When $m>0$ and $e=n(nm^2-1)$, the pair $ (nm,1)$ is a solution to the equation $\cP_{e}(n )$.

If $5\mid n$, we will show in the proof of Proposition~\ref{autfw*} that $\cP_{4e}(-5n)$ is not solvable.\ 
If $5\nmid n$, we can  choose $m$ such that $n(nm^2 - 1)\equiv \pm 2 \pmod 5$ and by reducing modulo 5, we see that the equation $\cP_{4e}(-5n)$ is then not solvable.\   

We can therefore apply Corollary~\ref{thhepw}.\ Since there are infinitely many such $m$, this concludes the proof, using Theorem~\ref{thcu} for the density statement.
\end{proof}

\appendix

\section{Automorphisms of special hyperk\"ahler fourfolds}\label{appB}

We determine  the group $\Aut(X)$ of biregular automorphisms  and the group $\Bir(X)$ of birational automorphisms  for some   hyperk\"ahler fourfolds $X$ of $\KKK^{[2]}$-type with Picard number 1 or 2.\ The case of Hilbert squares of very general polarized K3 surfaces is in Appendix~\ref{appA}.

Let $X$ be a hyperk\"ahler fourfold.\ 
There are natural morphisms
\begin{equation}\label{psiab}
\Psi_X^A\colon\Aut(X)\to O(H^2(X,\Z),q_{X})\quad\textnormal{and}\quad \Psi_X^B\colon\Bir(X)\to O(H^2(X,\Z),q_{X})
\end{equation}
which send a (birational) automorphism $\phi$ of $X$ to its action $\phi^*$ on cohomology (see~\cite[Proposition~25.14]{ghj} for  $\Psi_X^B$).\ 
Elements of $\Im(\Psi_X^A)$ preserve the nef cone $\Nef(X)$, elements of $\Im(\Psi_X^B)$ preserve the movable cone $\Mov(X)$, and both preserve the Picard lattice and the Hodge structure.\

The kernel of $\Psi_X^B$ 
  is contained in $\Aut(X)$, hence in the kernel of $\Psi_X^A$ 
  (\cite[Proposition~2.4]{ogu}).\ 
The group $\Ker(\Psi_X^A)$  is a finite group which is invariant by smooth deformations (\cite[Theorem~2.1]{hast4}) and is trivial for the Hilbert square of a K3 surface (\cite[Proposition~10]{bea}).\ 
It follows that {\em for any hyperk\"ahler fourfold $X$ of $\KKK^{[2]}$-type, both $\Psi_X^A$ and $\Psi_X^B$ are injective.}

\begin{prop}\label{prop27}
Let $X$ be a hyperk\"ahler fourfold corresponding to a very general point of a moduli space $\cM^{(\gamma)}_{2n}$.\ 
The group $\Bir(X)$ of birational automorphisms of $X$ is trivial, unless $n=1$, in which case  $\Aut(X)=\Bir(X)\isom\Z/2\Z$.
\end{prop}

\begin{proof} 
As we saw in Section~\ref{sec22}, the Picard group of $X$ is generated by the class $h$ of the polarization.\ 
Any birational automorphism leaves this class fixed, hence is in particular biregular of finite order.\  
Let $\phi$ be a non-trivial automorphism of $X$.\ 
Since $\phi$ extends to small deformations of $X$, the  restriction of $\phi^*$ to $h^\bot$ is a homothety whose ratio is, by~\cite[Proposition 7]{bea}, a root of unity; since it is real and non-trivial (by injectivity of $\Psi_X^A$), it must be $- \Id$.\ 
We will prove that such an isometry of $\Z h\oplus h^\bot$ does not extend to an isometry $\Phi$ of $H^2(X,\Z)$ unless $h^2=2n=2$.

When $ \gamma=1$, we may take $h = u+nv$, where $(u,v)$ is a standard basis for a hyperbolic plane $U$ contained in $H^2(X,\Z)$.\ 
Then, $u-nv$ is in $h^\bot$, hence the isometry $\Phi$, if it exists, must satisfy
\[
\Phi(u+nv)=u+nv\quad\textnormal{and}\quad \Phi(u-nv)=-u+nv,
\]
which yields  $2n\Phi(v)=2u$.\  This is possible only when $n=1$.\  Conversely, in the case $n=1$, the fourfold $X$ is a double EPW sextic and does carry a non-trivial involution (Example~\ref{qqq1}).\  Moreover, this involution is the only non-trivial automorphism of a very general double EPW sextic (see the end of the proof of~\cite[Proposition~B.9]{dk1}).

When $\gamma=2$ (so that $n\equiv -1\pmod4$), we let $\ell$ be an element of $H^2(X,\Z)$ orthogonal to $U$ and such that $\ell^2=-2$.\ We may take, as in the proof of Proposition~\ref{propirr}, $h=2 u+\frac{n+1}{2}v +\ell$, and $h^\bot$ contains $ v+\ell$ and $  u-\frac{n+1}{4}v$.\  The isometry $\Phi$  must then satisfy
\[
\Phi\bigl(2 u+\tfrac{n+1}{2}v +\ell\bigr)=2 u+\tfrac{n+1}{2}v +\ell\ ,\ \  \Phi(v+\ell)=-v-\ell\ ,  \ \ \Phi\bigl( u-\tfrac{n+1}{4}v\bigr)= -u+\tfrac{n+1}{4}v,
\]
hence $n\Phi(v)=4u+v+2\ell$; this is absurd since $n\ge 3$.
\end{proof}
 
\begin{rema}
The conclusion of the proposition does not necessarily hold if we  assume only that the Picard number of $X$ is 1.\  In fact, Proposition~\ref{prop27} is also proved in~\cite[Theorem~3.1]{sbcomp} and the proof given there implies that $\Bir(X)$ is trivial when the Picard number of $X$ is 1, unless $n\in\{1,3,23\}$.\  These three cases are actual exceptions: we just saw that all fourfolds corresponding to points of $\cM^{(1)}_2$ carry a non-trivial biregular involution; there   is   a 10-dimensional subfamily of 
$\cM^{(2)}_{6}$ whose elements consists of fourfolds that have a biregular automorphism of order~3 and whose very general elements have  Picard number~1 (\cite[Section~7.1]{sbcomp}); there is a (unique) fourfold in $\cM^{(2)}_{46}$ with Picard number~1 and a biregular automorphism of order 23 (\cite[Theorem~1.1]{sbisom}).
\end{rema}

We now turn our attention to the polarized hyperk\"ahler fourfolds studied in Example~\ref{ex:NefConesFano}.

\begin{prop}\label{autfw*}
Let $n$ be a positive square-free integer such that $n\equiv -1 \pmod 4$.\ Let $(X,H)$ be a polarized hyperk\"ahler fourfold of $\KKK^{[2]}$-type of degree $2n$ and divisibility 2, such that $\Pic(X)=\Z H\oplus \Z L$, with intersection matrix   $\left(\begin{smallmatrix}2n&0\\0&-2e'\end{smallmatrix}\right)$.\ Set $e:=e'n$.
 
  \noindent{\rm (a)} If neither equations $\cP_{e}(-n)$ and $\cP_{4e}(-5n)$ are   solvable and $e$ is not a perfect square, the groups   $\Aut(X)$ and   $\Bir(X)$ are equal.\ They are infinite cyclic, except when the equation $\cP_{e}(n)$ is solvable, in which case these groups are isomorphic to the infinite dihedral group $\Z\rtimes \Z/2\Z$.

  \noindent{\rm (b)} If the equation $\cP_{e}(-n)$ is not solvable but the equation $\cP_{4e}(-5n)$ is, the  group $\Aut(X)$ is trivial and the group $\Bir(X)$ is infinite cyclic, except when the equation $\cP_{e}(n)$ is solvable, in which case it is infinite dihedral.

  \noindent{\rm (c)} If the equation $\cP_{e}(-n)$ is solvable or if $e$ is a perfect square, the group $\Bir(X)$ is trivial.
\end{prop}
    
\begin{proof}
We saw that the map $\Psi_X^A\colon \Aut(X)\to  O(H^2(X,\Z))$ is injective.\  Its image consists of isometries which  preserve $\Pic(X)$ and the  ample cone and, since $b_2(X)-\rho(X)$ is odd,  restrict  to $\pm \Id $ on $\Pic(X)^\bot$ (\cite[proof of Lemma~4.1]{ogu5}).\  Conversely, by the Torelli  Theorem~\ref{tor1}, any isometry with these properties is in the image of $\Psi_X^A$.\  We begin with some general remarks on the group $G$ of isometries of $H^2(X,\Z)$ which preserve $\Pic(X)$ and the components of the positive cone, and restrict to $\eps \Id $ on $\Pic(X)^\bot$, with $\eps\in\{-1,1\}$.

The orthogonal group  of the rank-2 lattice $(\Pic(X),q_X)\isom I_1(2n)\oplus I_1(-2e')$  is easily  determined: if we let $\delta :=\gcd(n,e')$ and we write $n=\delta n'$ and $e'=\delta e''$,  we have
\[
O(\Pic(X),q_X)=\left\{ \begin{pmatrix} a& \alpha e''b\\n'b&\alpha a \end{pmatrix}  \Big|\ a,b\in \Z, \  a^2-n'e''b^2=1 ,\ \alpha\in\{-1,1\} \right\}.
\]
Note that $\alpha$ is the determinant of the isometry and
\begin{itemize}
 \item  such an isometry   preserves  the components of  the positive cone if and only if $a>0$; we denote the corresponding subgroup by $O^+(\Pic(X))$;
 \item when $e$ is not a perfect square, the group $SO^+(\Pic(X))$ is infinite cyclic, generated by the isometry $R$ corresponding to the minimal solution to the equation $\cP_{n'e''}(1)$ and the group $O^+(\Pic(X))$ is infinite dihedral;
 \item when $e$ is  a perfect square, so is $n'e''=e/\delta^2$, and   $O^+(\Pic(X))=\{ \Id, \left(\begin{smallmatrix} 1&0\\0&-1\end{smallmatrix}\right)\}$. 
\end{itemize}
    
As we saw during the proof of Proposition~\ref{propirr}, there exist standard bases $(u_1,v_1)$ and $(u_2,v_2)$ for two orthogonal hyperbolic planes in $\L_{\KKK^{[2]}}$, a generator $\ell$ for the $I_1(-2)$ factor, and an isometric identification $H^2(X,\Z)\isomto \L_{\KKK^{[2]}}$ such that
\[
H=2u_1+\frac{n+1}{2}v_1+\ell \quad \text{and}\quad L=u_2-e'v_2.
\]
The elements $\Phi$ of $G$ must then have $a>0$ and satisfy
\begin{eqnarray*} 
  \Phi(2u_1+\tfrac{n+1}{2}v_1+\ell)&=&a(2u_1+\tfrac{n+1}{2}v_1+\ell)+n'b(u_2-e'v_2)\\
  \Phi(u_2-e'v_2)&=&\alpha e''b(2u_1+\tfrac{n+1}{2}v_1+\ell)+\alpha a(u_2-e'v_2)\\
  \Phi(v_1+\ell)&=&\eps (v_1+\ell)\\
  \Phi(u_1-\tfrac{n+1}{4}v_1)&=&\eps (u_1-\tfrac{n+1}{4}v_1)\\
 \Phi(u_2+e'v_2)&=&\eps (u_2+e'v_2)
\end{eqnarray*}
(the last three lines correspond  to vectors in $\Pic(X)^\bot$).\ 
From this, we deduce
\begin{eqnarray*} 
   n\Phi(v_1)&=&2(a-\eps)u_1+\bigl( (a+\eps)\tfrac{n+1}{2}-\eps\bigr)v_1+(a-\eps)\ell+n'b(u_2-e'v_2)\\
  2 \Phi(u_2)&=&2\alpha e'' bu_1+\alpha e'' b\tfrac{n+1}{2}v_1+
  \alpha e''b \ell+(\eps+\alpha a)u_2+e' (\eps-\alpha a)v_2\\
   2e'\Phi(v_2)&=&-2\alpha e'' bu_1-\alpha e'' b\tfrac{n+1}{2}v_1
  -\alpha e''b \ell+(\eps-\alpha a)u_2+e' (\eps+\alpha a)v_2.
\end{eqnarray*}
From the first equation, we get $\delta\mid b$ and $a\equiv \eps\pmod{n}$; from the second equation, we deduce that $e''b$ and $\eps+\alpha a$ are even; from the third equation, we get $2\delta\mid b$ and $a\equiv \alpha\eps\pmod{2e'}$.\ All this is equivalent to $a>0$ and
\begin{equation}\label{eqab**}
2\delta\mid b\quad,\quad a\equiv \eps\pmod{n} \quad,\quad   a\equiv \alpha\eps\pmod{2e'}.
\end{equation}

Conversely, if these conditions are realized, one may define $\Phi$ uniquely on $\Z u_1\oplus \Z v_1\oplus \Z u_2\oplus \Z v_2 \oplus \Z \ell$ using  the formulas above, and extend it by $\eps \Id$ on the orthogonal of this lattice in $ \L_{\KKK^{[2]}}$ to obtain an element of $G$.

The first congruence in \eqref{eqab**} tells us that the identity on $\Pic(X)$ extended by $- \Id$ on its orthogonal does not lift to an isometry of $H^2(X,\Z) $.\  This means that the restriction $G\to O^+(\Pic(X))$ is injective.\  Moreover, the two congruences in \eqref{eqab**} imply $a\equiv \eps \equiv \alpha\eps\pmod{\delta}$.\  If $\delta>1$, since $n$, hence also $\delta$, is odd, we get  $\alpha=1$, hence the image of $G$ is contained in $SO^+(\Pic(X))$.

\noindent{\bf Assume $\alpha=1$.}\ The relations \eqref{eqab**} imply that $a-\eps$ is divisible by $n$ and $2e'$, hence by their least common multiple $2\delta n'e''$.\ We write
$b=2\delta b'$ and $a=2\delta n'e''a'+\eps$  
and
 obtain from the equality  $a^2-n'e''b^2=1$  the relation
\[
4\delta^2n^{\prime 2}e^{\prime\prime 2}a^{\prime 2}+4\eps \delta n'e''a'=4\delta^2n'e'' b^{\prime 2},
\]
hence 
\[
\delta n'e''a^{\prime 2}+\eps a'=\delta b^{\prime 2}.
\]
In particular, $a'':=a'/\delta$ is an integer and $b^{\prime 2}=a''(ea''+\eps)$.\ 

Since $a>0$ and $a''$ and $ea''+\eps$ are coprime, both are perfect squares and there exist coprime integers $r$ and $s$, with $r>0$, such that
\[
a''=s^2\quad,\quad e a''+\eps=r^2 \quad,\quad b'=rs.
\]

Since  $-1$ is not a square modulo $n$, we obtain $\eps=1$; the pair $(r,s)$ satisfies the Pell equation $r^2- e s^2=1$, and $a=2es^2+1$ and $b=2\delta rs$.\ In particular, either $e$ is not a perfect square and there are always infinitely many solutions, or $e$ is a perfect square and we get $r=1$ and $s=0$, so that $\Phi=\Id$.

\noindent{\bf Assume $\alpha=-1$.}\ As observed before, we have $\delta=1$, \ie, $n$ and $e'$ are coprime.\ Using~\eqref{eqab**}, we may write $b=2b'$ and $a=2a'e'-\eps$.\ Since $2\nmid n$ and $a\equiv \eps \pmod{n}$, we deduce  $\gcd(a',n)=1$.\ 
Substituting  into the equation $a^2-ne'b^2=1$, we obtain
\[
a'(e'a'-\eps)=nb^{\prime 2},
\]
hence there exist coprime integers $r$ and $s$, with $r\ge 0$, such that $b'=rs$, $a'=s^2$, and $e'a'-\eps=nr^2$.\  The pair $(r,s)$ satisfies the equation $nr^2-e's^2=-\eps$, and $a=2e's^2-\eps$ and $b=2rs$.\  In particular,  one of the two equations  $\cP_e(\pm n)$ is solvable.\  Note that at most one of the equations $\cP_e(\pm n)$ may be solvable: if $\cP_e(-\eps n)$ is solvable, $\eps e'$ is a square modulo $n$, while $-1$ is not.\ These isometries are all reflections and, since $n\ge 2$ and $e'\geq2$, $\left(\begin{smallmatrix} 1&0\\0&-1\end{smallmatrix}\right)$ is not one of them.\ In particular, if $e$ is a perfect square,   $G=\{\Id\}$.

\smallskip

We now go back to the proof of the proposition.\ We proved that the composition $\Aut(X)\to G\to O^+(\Pic(X))$ is injective and so is the morphism $\Bir(X)\to G\to O^+(\Pic(X))$ (any element of its kernel is in $\Aut(X)$).

Under the hypotheses of (a), both slopes of the nef cone are irrational (Example~\ref{ex:NefConesFano}).\  By~\cite[Theorem 1.3]{ogu}, the groups $\Aut(X)$ and $\Bir(X)$ are then equal and infinite.\  The calculations above allow us to be more precise: in this case, the ample cone is just one component of the positive cone and the groups $\Aut(X)$ and $G$ are isomorphic.\  The proposition then follows from the discussions above (note that when there are involutions, the equation  $\cP_e(  n)$ has a solution $(nr,s)$ hence, in the notation above, $\eps=-1$ and these involutions act on $H^2(X,\Z)$ as the symmetries about ample square-2 classes $ rH+sL$).

Under the hypotheses of (c), the slopes of the extremal rays of the nef and movable cones are rational (Example~\ref{ex:NefConesFano}) hence, by~\cite[Theorem~1.3]{ogu} again, $\Bir(X)$ is a finite group.\  By~\cite[Proposition~3.1(2)]{ogu},  any non-trivial element $\Phi$ of its image in $O^+(\Pic(X))$ is an involution which satisfies $\Phi(\Mov(X))= \Mov(X)$, hence switches the two extremal rays of this cone.\  This means $\Phi(H\pm \mu_{n,e}L)=H\mp \mu_{n,e}L$, hence $\Phi(H)=H$, so that  $\Phi=\left(\begin{smallmatrix} 1&0\\0&-1\end{smallmatrix}\right)$.\  Since we saw that this is impossible, the group  $\Bir(X)$ is trivial.

Under the hypotheses of (b), the slopes of the nef cone are both rational and the slopes of the movable cone are both irrational (Example~\ref{ex:NefConesFano}).\ By~\cite[Theorem~1.3]{ogu} again, $\Aut(X)$ is a finite group and $\Bir(X)$ is infinite.\ 
The same reasoning as in case~(c) shows that the group  $\Aut(X)$ is in fact trivial; moreover, the group $\Bir(X)$ is a subgroup of $\Z$, except when the equation $\cP_{e}(n) $ is solvable, where it is a subgroup of $\Z\rtimes \Z/2\Z$.\

In the latter case, such an infinite subgroup is isomorphic  either to $\Z$ or to $\Z\rtimes \Z/2\Z$ and we exclude the first case by  showing that there is indeed a regular involution on  a birational model of $X$ (this generalizes the case $n=3$ and $e=6$ treated in~\cite{hast3}).\ We denote by $(na_n,b_n)$   the minimal solution to the equation $\cP_{e}(n) $ and set $x_n:=na_n+b_n\sqrt{e}\in\Z[\sqrt{e}]$.

As observed in Remark~\ref{rmk:NefConeHK4}(b), the set of all positive solutions $(a,b)$ to the equation $na^2-4e'b^2=-5$ (so that $aH\pm 2bL\in \cFlop_X $, or equivalently, $(na,b)$ is a solution to the equation $\cP_{4e}(-5n)$)
  determines an infinite sequence of rays $\R_{\ge0}(2e'bH\pm naL)$ in $\Mov(X)$ and the nef cones of hyperk\"ahler fourfolds birational to $X$ can be identified with the chambers with respect to this collection of rays.\ For example, if $(na_{5n},b_{5n})$ is the minimal solution to the equation $\cP_{4e}(-5n)$, the two extremal rays of the cone $\Nef(X)$ are spanned by $\alpha_0:=2e'b_{-5n}H - na_{-5n}L$ and $\alpha_1:=2e'b_{-5n}H + na_{-5n}L$.\ We want to describe all solutions $(a,b)$.
  
  \begin{lemm}
  The minimal solution to the Pell equation  $\cP_{e}(1)$ is given by $y_1=n a_n^2 + e'b_n^2 +2a_nb_n \sqrt{e}$ and  
  all the   solutions  $(na,b)$  to the equation $\cP_{4e}(-5n)$ are   given by the two disjoint families
  $$na + 2b \sqrt{e}=\pm x_{-5n}y_1^m\ \textnormal{ or }\ \pm \overline{x_{-5n}} y_1^m,\quad m\in\Z
  ,$$
  where $x_{-5n}:=na_{-5n} +2b_{-5n} \sqrt{e}$.
\end{lemm}

\begin{proof}
Let $(a,b)\in\Z^2$ and set $x:=na +b \sqrt{e}\in\Z[\sqrt{e}]$ and $y:=\frac{1}{n}xx_n$.\ We have
$$y = naa_n +e'bb_n +(ab_n+ a_nb)\sqrt{e}=:a'+b'\sqrt{e} \in\Z[\sqrt{e}]
$$
and, if $N$ is the norm in the ring $\Z[\sqrt{e}]$, we have $N(x_n)=n$ and, if $t$ is any non-zero integer,
$$(na,b) \textnormal{ solution to }\cP_{e}(tn)\ \Leftrightarrow\ N(x)=tn\ \Leftrightarrow\ N(y)=t\ \Leftrightarrow\  (a',b')\textnormal{ solution to }\cP_{e}(t).
$$
Since $x=y\overline{ x_n}$, this establishes a one-to-one correspondance between the solutions of the equation $\cP_{e}(tn)$ and those of $\cP_{e}(t)$.\ In particular, the minimal solution to the Pell equation  $\cP_{e}(1)$ is given by $y_1=\frac{1}{n} x_n^2=n a_n^2 + e'b_n^2 +2a_nb_n \sqrt{e}$.

The solutions to the equation $\cP_{e}(-5)$ were analyzed in~\cite[Theorem~110]{nage}: if $y_{-5}$ corresponds to its minimal solution, they are all given by $ \pm y_{-5}y_1^m$, 
$m\in\Z$, and their conjugates.\ It follows that all the   solutions $(na,b)$ to the equation $\cP_{ e}(-5n)$ are given by $\pm y_{-5}\overline{ x_n}y_1^m$ and their conjugates.\ Since the ``imaginary'' part of $y_1$ is even and its ``real'' part is odd, the parity of the ``imaginary'' parts of these solutions are all the same.\ Since the equation $\cP_{4e}(-5n)$ is solvable, they are all even, and we therefore obtain all the solutions to the equation  
$\cP_{4e}(-5n)$.

To prove that the conjugates provide a disjoint set of   solutions, we need to check, by~\cite[Theorem~110]{nage}, that $5$  does not divide  $4e$.

Assume first  $5\mid e'$.\ Since the equation $\cP_{e}(n)$ is solvable, we have  $\left(\frac{n}{5} \right)=1$; moreover, since $n\equiv -1\pmod{4}$, we have   $\left(\frac{e'}{n}\right)=-1$.\ The solvability of the equation $\cP_{4e}(-5n)$ implies $\left(\frac{5}{n}\right)=\left(\frac{e'}{n}\right)$; putting all that together contradicts quadratic reciprocity. 

Assume now  $5\mid n$ and set $n':=n/5$.\  Since the equation $\cP_{e}(n)$ is solvable, we have $\left(\frac{e'}{5} \right)=1$; moreover, since $n'\equiv -1\pmod{4}$, we have $\left(\frac{e'}{n'}\right)=-1$.\  Since $5\nmid e'$, the equation $\cP_{n',20e'}(-1)$ is solvable, hence $\left(\frac{5e'}{n'}\right)=1$; again, this contradicts quadratic reciprocity.
\end{proof}

 We can reinterpret this as follows.\  Since $\gcd(n,e')=1$ (because $na_n^2-e'b_n^2=1$), the generator $R$ of the group  $SO^+(\Pic(X))$ previously defined is $R=\left(\begin{smallmatrix} n a_n^2 + e'b_n^2&  2e'a_nb_n\\2n a_nb_n& n a_n^2 + e'b_n^2 \end{smallmatrix}\right)$.\ If we set $\alpha_{i+2}:=R(\alpha_i)$, the lemma means that the infinitely many rays in $\Mov(X)$ described above are   the $(\R_{\ge0}\alpha_i)_{i\in\Z}$.\ The fact that  the conjugate solutions form a disjoint family means exactly that the ray $\R_{\ge0}\alpha_2$ is ``above'' the ray $\R_{\ge0}\alpha_1$; in other words, we have an ``increasing'' infinite sequence of  rays 
 $$\cdots< \R_{\ge0}\alpha_{-1}<\R_{\ge0}\alpha_0< \R_{\ge0}\alpha_1< \R_{\ge0}\alpha_2< \cdots.$$

It follows from the discussion above that the reflection $R\left(\begin{smallmatrix} 1&0\\0&-1\end{smallmatrix}\right)$  belongs to the group $G$ and preserves the nef cone of the birational model $X'$ of $X$ whose nef cone is generated by $\alpha_1$ and $\alpha_2$.\  It is therefore induced by a biregular involution of $X'$ which defines a birational involution of $X$.\  This concludes the proof of the proposition.
\end{proof}

\begin{rema}\label{rmk:TwoBirationalModels}
It follows from the  proof above that in case (b), if both equations $\cP_{4e}(-5n)$ and $\cP_{e}(n)$ are solvable, $X$ has exactly one non-trivial birational model.\  It is obtained from $X$ by a composition of Mukai flops with respect to  Lagrangian planes (Remark~\ref{rmk:NefConeHK4}(b)).
\end{rema}

\section{Automorphisms of Hilbert squares of very general K3 surfaces}\label{appA}

Since the extremal rays of  the movable cone of the Hilbert square $S^{[2]}$ of a very general  K3 surface $S$ of given degree $2e $ are   rational (Example~\ref{ex:NefConesHilbertSchemes}), its group $\Bir(S^{[2]}) $ of birational automorphisms is finite (\cite[Theorem 1.3(2)]{ogu}).\ 
Using the Torelli Theorem~\ref{tor1}, one can determine the group $\Aut(S^{[2]})$ of its biregular automorphisms (\cite[Theorem~1.1]{bcs}) and also, using the description of the nef and movable cones (Example~\ref{ex:NefConesHilbertSchemes}),  the group $\Bir(S^{[2]}) $.

\begin{theo}[Boissi\`ere--Cattaneo--Nieper-Wi\ss kirchen--Sarti]\label{bbb} 
Let $(S,L)$ be a polarized K3 surface of degree $2e$ with Picard group $\Z L$.\ 
The variety $S^{[2]}$ has a non-trivial automorphism if and only if either $e=1$, or  the equation $\cP_{e}(-1)$ is solvable and the equation $\cP_{4e}(5)$ is not.
\end{theo}

The non-trivial  automorphism in Theorem~\ref{bbb} is then unique and an anti-symplectic involution.\ 
When $e\ge 2$, this involution acts on $H^2(S^{[2]} ,\Z)$ as the symmetry $s_D$ about the line spanned by the square-2 class $D:=b_{-1} L_2  - a_{-1}\delta$, where $(a_{-1},b_{-1})$ is the minimal solution of the  equation $\cP_e(-1)$.\ 
When $e=1$, this involution is induced by an involution of $(S,L)$ and its acts on $H^2(S^{[2]} ,\Z)$ as the symmetry about the plane $\Pic(S^{[2]})$.

\begin{exam}\label{sma}
Theorem~\ref{bbb} applies for example for $e=m^2+1$ with $m\ne 2$, or $e=13$.\
When $e=2$, the surface $S$ is a quartic in $\P^3$ which contains no lines nor conics, and the involution $\sigma$ of $S^{[2]}$ is the Beauville involution: it sends a pair of points in $S$ to the residual intersection with $S$ of  the line that they span.\ We have   $D=L_2-\delta$ and $s_D(L_2)= 3L_2-4\delta$ (\cite[Th\'eor\`eme 4.1]{deb},~\cite[Section 6.1]{bcs}); the quotient $S^{[2]}/\sigma$ is a triple cover of the Pl\"ucker quadric $\Gr(2,4)\subset \P^5$.\ When $e\ge 3$, the quotient  $S^{[2]}/\sigma$ is an EPW sextic (Corollary~\ref{cor53}). 
\end{exam}

\begin{prop}\label{bbb2}
Let $(S,L)$ be a polarized K3 surface of degree $2e$ with Picard group $\Z L$.\  The group $\Bir(S^{[2]})$ is trivial except in the following cases:
\begin{itemize}
\item  $e=1$, or the equation $\cP_{e}(-1)$ is solvable and the equation $\cP_{4e}(5)$ is not, in which cases $\Aut(S^{[2]})=\Bir(S^{[2]})\isom\Z/2\Z$;
\item $e>1$, and $e=5$ or $5\nmid e$, and both equations  $\cP_{e}(-1)$ and $\cP_{4e}(5)$ are solvable, in which case $\Aut(S^{[2]})=\{\Id\}$ and $\Bir(S^{[2]})\isom\Z/2\Z$.\footnote{There are  cases where both equations $\cP_{e}(-1)$ and $\cP_{4e}(5)$ are solvable and $5\nmid e$; for example, $e=29$.}
\end{itemize}
\end{prop}
 
In the second case, there is a difference between the case $e=5$ and the case $5\nmid e$: when $5\nmid e$, there is a hyperk\"ahler fourfold (in fact, a double EPW sextic) birational to $S^{[2]}$ on which the involution is biregular, but not when $e=5$.

\begin{proof}
If $\phi\in \Bir(S^{[2]})$ is not biregular, $  \phi^* $ acts on the movable cone $\Mov(S^{[2]})$ in such a way that $ \phi^*(\Amp(S^{[2]}))\cap \Amp(S^{[2]})=\vide$.\ This implies $\Mov(S^{[2]})\ne \Nef(S^{[2]})$ hence, by Example~\ref{ex:NefConesHilbertSchemes}, the equation $\cP_{4e}(5)$ has a minimal solution $(a_5,b_5)$.\  By Theorem~\ref{bbb}, the group $\Aut(S^{[2]})$ is then trivial.\

Moreover, $\phi^*$ maps one extremal ray of the movable cone (spanned by $L_2$) to the other extremal ray (spanned by the primitive vector $a_1L_2-eb_1\delta$).\ Therefore, we have $ \phi^*(L_2)= a_1L_2-eb_1\delta$ and, by applying this relation to $\phi^{-1}$, also $ \phi^*(a_1L_2-eb_1\delta)=L_2$.\ 
This implies that  $ \phi^*$ is a   completely determined involution of $\Pic(S^{[2]})$.\ 
In particular, $ \phi^2 $ is an automorphism, hence is trivial: $\phi$ is an involution.

The transcendental lattice   $\Pic(S^{[2]} )^\bot \subset H^2( S^{[2]},\Z)$ carries a simple rational Hodge structure.\footnote{This is a classical fact found for example  in~\cite[Lemma~3.1]{huyk3}\label{fsimple}.}\ 
Since the eigenspaces of the involution $\phi^*$ of $ H^2( S^{[2]},\Z)$ are sub-Hodge structures, the restriction of $\phi^*$ to $\Pic(S^{[2]} )^\bot$ is $\eps \Id$, with $\eps\in\{-1, 1\}$.\ On $\Pic(S^{[2]} )$, we saw that $\phi^*$ has matrix $\left(\begin{smallmatrix}a_1& b_1 \\-eb_1&-a_1\end{smallmatrix}\right)$ in the basis $(L_2,\delta)$.\ The extension from $ \Pic(S^{[2]} )\oplus \Pic(S^{[2]} )^\bot$ to the overlattice $H^2( S^{[2]},\Z)$ of such an involution can be studied as in the proof of Proposition~\ref{autfw*} (see also~\cite[Lemma 5.2]{bcs} when $\eps=-1$).\ 
The conclusion is that  there exist positive integers $r$ and $s$ such that $r^2-es^2=\eps$  and $a_1+b_1\sqrt{e}=(r+s\sqrt{e})^2$.\ The value $\eps=1$  would contradict the minimality of the solution $(a_1,b_1)$ to the equation $\cP_e(1)$.\ Hence we have $\eps=-1$ and $(r,s)$ is the minimal solution to the equation $\cP_e(-1)$.\ 
In particular,  $\phi^*$ is a   completely determined involution of $H^2( S^{[2]},\Z)$ and, since $\Psi_X^B$ is injective, $\Bir(S^{[2]})$ has at most 2 elements.

By~\cite[Theorem~110]{nage}, the solutions to the equation $\cP_{4e}(5) $ are all given by $\pm x_5x_1^m$, $m\in\Z$, where 
 $x_5:= a_5+2b_5\sqrt{e}$ and $x_1:= a_1+  b_1\sqrt{e}$.\ The associated positive elements of $\Z[\sqrt{e}]$ are ordered as follows
\[
\cdots <x_5x_1^{-2}\le\overline{ x_5}x_1^{-1}<x_5x_1^{-1}\le \bar x <\sqrt{5}<x_5\le \overline{ x_5}x_1<x_5x_1\le \overline{ x_5}x_1^2<x_5x_1^2<\cdots
\]
Still by~\cite[Theorem~110]{nage},
\begin{itemize}
\item[(a)] either $5\mid e$, the conjugate solutions are the same, and $x_5= \overline{ x_5}x_1$;
\item[(b)] or $5\nmid e$, the conjugate solutions are different, and $x<\overline{ x_5}x_1^{-1}$.\end{itemize}
Set $x_{-1}:= r+s\sqrt{e}$, so that $x_1=x_{-1}^2$.\ 
In  case (a), we have  $\overline{x_5\overline{ x_{-1}}}=-x_5\overline{ x_{-1}}$ and since $\overline{x_5\overline{ x_{-1}}}x_5\overline{ x_{-1}}=-5$, we obtain $e=5$.\ In that case,   there is indeed a non-trivial birational involution on $S^{[2]}$ (Example~\ref{exa311}).

In terms of the rays generated by $aL_2-2eb\delta$, where $(a,b)$ is a solution to the equation $\cP_{4e}(5) $, multiplying $x=a+2b\sqrt{e}$ by $x_1$ corresponds to applying the rotation $R:=\left(\begin{smallmatrix}
   a_1&-b_1\\-eb_1&a_1
   \end{smallmatrix}\right)$, which sends the extremal ray $\R_{\ge0}L_2$ of the movable cone to its other extremal ray $\R_{\ge0}(a_1L_2-eb_1\delta)$; the operation $x\mapsto \bar xx_1 $ therefore corresponds to applying the reflection $R\left(\begin{smallmatrix}
   1&0\\0&-1
   \end{smallmatrix}\right)=\left(\begin{smallmatrix}
   a_1&b_1\\-eb_1&-a_1
   \end{smallmatrix}\right)$, which is the symmetry  $s_D$ about the line spanned by the class $D:=b_{-1} L_2  - a_{-1}\delta$

Geometrically, the situation is clear  in  case (b):  we have exactly two rays inside $ \Mov(S^{[2]})$ which are symmetric about the line $\R D$.\  As explained in Remark~\ref{rmk:NefConeHK4}(b),  the  three associated  chambers  correspond  to the hyperk\"ahler fourfolds birational to $S^{[2]}$: the ``middle one'' corresponds to a fourfold $X $ whose nef cone  is  preserved by the involution $s_D$.\ There is a biregular involution on $X $ which induces a birational involution on~$S^{[2]}$.
\end{proof}

\begin{exam}[The O'Grady involution]\label{exa311}
A general polarized K3 surface $S$ of degree 10  is the transverse intersection of the Grassmannian $\Gr(2,\C^5)\subset \P(\bbw2\C^5)=\P^9$, a quadric $Q\subset \P^9$, and a $\P^6\subset \P^9$.\   
A general point of $S^{[2]}$ corresponds to $V_2,W_2\subset V_5$.\ Then 
$$\Gr(2,V_2\oplus W_2)\cap S=\Gr(2,V_2\oplus W_2)\cap Q\cap \P^6 \cap \bbw2(V_2\oplus W_2))\subset\P^2
$$
is the intersection of two general conics in $\P^2$ hence consists of 4 points, including $[V_2]$ and $[W_2]$.\ The (birational) O'Grady   involution $S^{[2]}\dra S^{[2]}$ takes the pair of points $([V_2],[W_2])$  to the residual two points of this intersection.
\end{exam}

\end{document}